\documentclass{article}
\DeclareUnicodeCharacter{03C3}{\ensuremath{\sigma}}
\usepackage[utf8]{inputenc}
\usepackage{amsmath}
\usepackage{amsfonts}
\usepackage{amssymb}
\usepackage{graphicx}%
\providecommand{\U}[1]{\protect\rule{.1in}{.1in}}
\newtheorem{theorem}{Theorem}
\newtheorem{acknowledgement}[theorem]{Acknowledgement}

\newtheorem{claim}[theorem]{Claim}

\newtheorem{conjecture}[theorem]{Conjecture}
\newtheorem{corollary}[theorem]{Corollary}

\newtheorem{definition}[theorem]{Definition}

\newtheorem{lemma}[theorem]{Lemma}

\newtheorem{problem}[theorem]{Problem}
\newtheorem{proposition}[theorem]{Proposition}

\newenvironment{proof}[1][Proof]{\noindent\textbf{#1.} }{\ \rule{0.5em}{0.5em}}
\begin{document}

\date{}
\title{The Category Dichotomy for Ideals}
\author{Alan Dow \thanks{\textit{keywords: }Ideals on countable sets, Borel ideals,
Category Dichotomy for ideals, Kat\v{e}tov order, cardinal invariants.
}
\and Ra\'{u}l Figueroa-Sierra\thanks{ The second author was partially supported by
La Academia de Ciencias de América Latina-Programa de subvenciones a investigadores 
jóvenes y estudiantes de postgrado venezolanos y colombianos en laboratorios de la 
Región Latinoamericana (2023). The third author was supported
by the PAPIIT grant IA 104124 and the CONAHCYT grant CBF2023-2024-903. The
fourth author was supported by d by a PAPIIT grant IN101323.}
\and Osvaldo Guzm\'{a}n
\and Michael Hru\v{s}\'{a}k}
\maketitle

\begin{abstract}
We prove that there is an ideal on $\omega$ that is not Kat\v{e}tov below
\textsf{nwd }and does not have restrictions above $\mathcal{ED}.$ We also
prove that in the Laver model every tall \textsf{P}-ideal is Kat\v{e}tov-Blass
above $\mathcal{ED}_{\text{\textsf{fin}}}$ and that it is consistent that
every \textsf{Q}$^{+}$ ideal is meager.

\end{abstract}

\section{Introduction}

The theory of ideals on countable sets is now a fundamental part of set theory
and has countless applications in infinite combinatorics. For this reason, it
is essential to have tools to classify and order the ideals. One such tool is
the \emph{Kat\v{e}tov order} (which was introduced in \cite{KatetovOriginal}
by Kat\v{e}tov to study convergence in topological spaces) which has proven to
be highly fruitful. When restricted to maximal ideals, it coincides with the
more well-known \emph{Rudin-Keisler order}, which has been greatly studied in
the literature (see \cite{UltrafiltersMathematics}). The Kat\v{e}tov order
becomes particularly interesting when restricted to ideals with a certain
degree of definability, such as Borel or analytic. As it is often the case,
when restricting to definable objects, we find a richer structure that would
not be expected to be present in the general case. One such result is the
\emph{Category Dichotomy}\footnote{We point out that the two alternatives of
the Category Dichotomy are not mutually exclusive.} of the fourth author,
which is the following\footnote{The undefined notions will be reviewed in the
next section.}:

\begin{theorem}
[H. \cite{KatetovOrderonBorelIdeals}]\textbf{Category Dichotomy for Borel
Ideals}\emph{ }\label{Dicotomia Cat}

Let $\mathcal{I}$ be a Borel ideal on $\omega.$ One of the following holds:

\begin{enumerate}
\item $\mathcal{I}\leq_{\text{\textsf{K}}}$ \textsf{nwd.}

\item There is $X\in\mathcal{I}^{+}$ such that $\mathcal{ED}\leq
_{\text{\textsf{K}}}\mathcal{I}\upharpoonright X.$
\end{enumerate}
\end{theorem}

We can now ask whether it is possible to extend this dichotomy to a larger
class of ideals. The first question is whether the dichotomy holds for all of
them. In other words:\qquad\qquad\qquad\ 

\begin{problem}
Is every ideal on $\omega$ either Kat\v{e}tov below \textsf{nwd }or has a
restriction above $\mathcal{ED}?$ \qquad\ \ \ \ \ \ \ 
\end{problem}

One would expect the answer to be negative, and it is easy to produce such
examples assuming the \emph{Continuum Hypothesis} (\textsf{CH}) \ However,
finding an answer within \textsf{ZFC} is much harder. The main result of this
article is to provide such an example. Upon examining the proof of the Theorem
\ref{Dicotomia Cat}, it is evident that the argument can be extended to all
ideals if we assume certain determinacy for transfinite games
(\textsf{AD($\mathbb{R}$) }is more than enough). Furthermore, the third author
and Jareb Navarro recently proved that the Category Dichotomy holds for all
ideals in the Solovay model. In this way, the \emph{Axiom of Choice
(\textsf{AC})} must play a role if we are to construct a counterexample to the
dichotomy. While looking for the example, we explored the following classes of ideals:

\begin{enumerate}
\item Maximal ideals.

\item Ideals generated by \textsf{MAD }families.

\item \textsf{P}-ideals.

\item Ideals induced by independent families.

\item Ideals of nowhere dense sets of countable topological spaces.
\end{enumerate}

\qquad\qquad\qquad

We will prove that the desired example can be found in the last class; there
is a countable topological space $X$ \ such that the ideal of nowhere dense
sets of $X$ does not satisfy the Category Dichotomy. Such space already
appears in the literature, it is the space constructed by the first author in
\cite{PiWeightandFrechetProperty}, which he used to answer a question of
Juh\'{a}sz. \bigskip

Given $\Gamma$ a class of ideals on $\omega,$ we will say that the
\emph{Category Dichotomy holds for }$\Gamma$ if every ideal on $\Gamma$ is
either below \textsf{nwd }or has a restriction above $\mathcal{ED}.$ The
following table summarizes our results regarding the classes of ideals
mentioned before:

\qquad\qquad\qquad\ \ \ 

\begin{center}%
\begin{tabular}
[c]{|c|c|}\hline
\textbf{Class of ideals} & \textbf{Category Dichotomy}\\\hline
& \\
Borel & True\\
& \\
Maximal ideals & Independent\\
& \\
Ideals generated & Consistently false, unknown\\
by \textsf{MAD} families & if consistently true\\
& \\
\textsf{P}-ideals & Independent\\
& \\
Non-meager ideals & Independent\\
& \\
Nowhere dense ideals induced & Consistently false, unknown\\
by independent families & if consistently true\\
& \\
Ideals of nowhere dense sets of & False\\
countable topological spaces & \\\hline
\end{tabular}
\smallskip

\textsf{Table 1. The Category Dichotomy for some classes of ideals}
\end{center}

\qquad\qquad\ \ \textsc{\ }

The paper is organized in the following way: first we review all the
definitions and results regarding the Kat\v{e}tov order and definable ideals
that will be used in the paper. In Section \ref{seccion non} we introduce the
\emph{uniformity number for a class of ideals}$.$ Concretely, if $\Gamma$ is a
class of ideals on $\omega,$ by \textsf{non}$\left(  \Gamma\right)  $ we
denote the least cofinality of an ideal that is not in $\Gamma.$ We review
what is known regarding this invariants and make some small remarks$.$
Sections \ref{seccion ultrafiltro} and \ref{seccion MAD} are respectively
devoted to maximal ideals and ideals generated by \textsf{MAD }families. We
prove that the Category Dichotomy for maximal ideals is equivalent to the non
existence of Ramsey ultrafilters and the Category Dichotomy for ideals
generated by \textsf{MAD }families is equivalent to the non existence of Cohen
indestructible maximal almost disjoint families. Section \ref{seccion P ideal}
is dedicated to \textsf{P}-ideals. We prove that the Category Dichotomy for
this class fails if either $\mathfrak{t<}$ \textsf{cov}$\left(  \mathcal{M}%
\right)  $ or $\mathfrak{t=d}.$ On the other hand, we show that it is true in
the Laver model. In Section \ref{no magro y Q mas} we prove that it is
consistent that every \textsf{Q}$^{+}$ ideal is meager, hence it is consistent
that every non-meager ideal satisfies the Category Dichotomy. In Section
\ref{seccion ideales nwd} we look at ideals of the form \textsf{nwd}$\left(
X\right)  $ where $X$ is a countable topological space. The desired
counterexample to the dichotomy can be found within this class. In Section
\ref{seccion familias independientes} we look at the nowhere dense ideals of
the spaces induce by independent families and prove that the Category
Dichotomy may fail for this class. 

\section{Preliminaries and notation}

For a set $X$, we denote by $\mathcal{P}\left(  X\right)  $ its power set. We
say that $\mathcal{I}\subseteq\mathcal{P}\left(  X\right)  $ is an \emph{ideal
on }$X$ if $\emptyset\in\mathcal{I}$ and $X\notin\mathcal{I},$ for every
$A,B\subseteq X,$ if $A\in\mathcal{I}$ and $B\subseteq A$ then $B\in
\mathcal{I}$ and if $A,B\in\mathcal{I}$ then $A\cup B\in\mathcal{I}.$ On the
other hand, $\mathcal{F\subseteq}$ $\wp\left(  X\right)  $ is a \emph{filter
on }$X$ if $X\in\mathcal{F}$ and $\emptyset\notin\mathcal{F},$ for every
$A,B\subseteq X,$ if $A\in\mathcal{F}$ and $A\subseteq B$ then $B\in
\mathcal{F}$ and if $A,B\in\mathcal{F}$ then $A\cap B\in\mathcal{F}.$ An
\emph{ultrafilter} is a maximal filter that does not contain any finite set
(for us, all ultrafilters are non principal). Given a family $\mathcal{B}$ of
subsets of $X,$ we define $\mathcal{B}^{\ast}=\left\{  X\setminus B\mid
B\in\mathcal{B}\right\}  .$ It is easy to see that if $\mathcal{F}$ is a
filter then $\mathcal{F}^{\ast}$ is an ideal (called the \emph{dual ideal of
}$\mathcal{F}$) and if $\mathcal{I}$ an ideal then $\mathcal{I}^{\ast}$ is a
filter (called the dual filter of $\mathcal{I}$). If $\mathcal{I}$ is an ideal
on $X$, we let $\mathcal{I}^{+}=\wp\left(  X\right)  \smallsetminus
\mathcal{I}$ be the family of $\mathcal{I}$\emph{-positive sets. }If
$\mathcal{F}$ is a filter, we define $\mathcal{F}^{+}$ $\mathcal{=}$ $\left(
\mathcal{F}^{\ast}\right)  ^{+};$ it is easy to see that $\mathcal{F}^{+}$ is
the family of all sets that have non-empty (infinite) intersection with every
element of $\mathcal{F}.$ If $A\in\mathcal{I}^{+}$ then \emph{the restriction
of }$\mathcal{I}$ to $A$, defined as $\mathcal{I\upharpoonright}A=\wp\left(
A\right)  \cap\mathcal{I}$, is an ideal on $A.$

\qquad\ \ \ \qquad\ \ 

The \emph{cardinal invariants of the continuum} will play a fundamental role
in this paper. We begin with a brief review of their definitions and key
properties, focusing on those relevant to our discussion. For a more
comprehensive treatment, we refer the reader to \cite{HandbookBlass} and
\cite{Barty}.

\qquad\qquad\ \ \ 

By $\mathfrak{c}$ we denote the size of the real numbers. Letting
$f,g\in\omega^{\omega}$, define $f\leq g$ if and only if $f\left(  n\right)
\leq g\left(  n\right)  $ for every $n\in\omega$ and $f\leq^{\ast}g$ if and
only if $f\left(  n\right)  \leq g\left(  n\right)  $ holds for all
$n\in\omega$ except finitely many. We say a family $\mathcal{B}\subseteq
\omega^{\omega}$ is \emph{unbounded }if $\mathcal{B}$ is unbounded with
respect to $\leq^{\ast}.$ A family $\mathcal{D}\subseteq\omega^{\omega}$ is a
\emph{dominating family }if for every $f\in\omega^{\omega},$ there is
$g\in\mathcal{D}$ such that $f\leq^{\ast}g.$ The \emph{bounding number
}$\mathfrak{b}$ is the size of the smallest unbounded family and the
\emph{dominating number }$\mathfrak{d}$ is the smallest size of a dominating
family. It is easy to see that $\mathfrak{d}$ is also the least size of a
dominating family with the more strict order $\leq$ (on the other hand, the
least size of an unbounded family with $\leq$ is simply $\omega$). The
invariant $\mathfrak{d}$ can also be characterized using \emph{interval
partitions }(partitions of $\omega$ into intervals). Given $P$ and $R$
interval partitions, define $P\leq R$ if every interval of $R$ contains at
least one interval of $P.$ It can be proved that the dominating number is
equal to the least size of a dominating family of interval partitions (there
is a similar characterization of $\mathfrak{b},$ but we do not mention it
since it will not be used in the text).

\qquad\qquad\qquad\qquad\ \ \ 

Letting $X$ be a topological space and $N\subseteq X,$ we say $N$ is
\emph{nowhere dense }if for every non empty open set $U\subseteq X,$ we can
find a non empty open set $V$ such that $V\subseteq U$ and $V\cap
N=\emptyset.$ It is very easy to see that the closure of a nowhere dense set
is also nowhere dense. A \emph{meager set }is a countable union of nowhere
dense sets. By \textsf{cov}$\left(  \mathcal{M}\right)  $ we denote the
smallest size of a family of meager sets that covers $\omega^{\omega}$.

\qquad\ \ \ \ 

For any two sets $A$ and $B,$ we say $A\subseteq^{\ast}B$ ($A$ is an almost
subset of $B$) if $A\setminus B$ is finite. For $\mathcal{H\subseteq}$
$\left[  \omega\right]  ^{\omega}$ and $A,B\subseteq\omega,$ we say that
$A\ $\emph{is a pseudointersection of }$\mathcal{H}$ if it is almost contained
in every element of $\mathcal{H}.$ On the other hand, $B$ \emph{is a
pseudounion of }$\mathcal{H}$ if it almost contains every element of
$\mathcal{H}.$ We say $\mathcal{T}=\left\{  A_{\alpha}\mid\alpha
<\kappa\right\}  $ is a \emph{tower }if it is $\subseteq^{\ast}$-decreasing
and has no infinite pseudointersection. On the other hand, we call
$\mathcal{S}$ $=\left\{  A_{\alpha}\mid\alpha<\kappa\right\}  $ an
\emph{increasing tower }if $\left\{  \omega\setminus A_{\alpha}\mid
\alpha<\kappa\right\}  $ is a tower. The \emph{tower number }$\mathfrak{t}$ is
the least length of a tower. It is known that $\mathfrak{t\leq b}$ and
$\mathfrak{d}$ is larger than both $\mathfrak{b}$ and \textsf{cov}$\left(
\mathcal{M}\right)  .$ Evidently, $\mathfrak{c}$ is larger than all the other
cardinal invariants.

\qquad\qquad\qquad\qquad\ \ \ 

We say that $T\subseteq\omega^{<\omega}$ is a \emph{tree }if it is closed
under taking initial segments. If $s\in T$ we define \textsf{suc}$_{T}\left(
s\right)  =\left\{  n\mid s^{\frown}n\in T\right\}  $ (where $s^{\frown}n$ is
the sequence that has $s$ as an initial segment and $n$ in the last entry). If
$T\subseteq\omega^{<\omega}$ we say that $f\in\omega^{\omega}$ is a
\emph{branch of }$T$ if $f\upharpoonright n\in T$ for every $n\in\omega.$ The
set of all branches of $T$ is denoted by $\left[  T\right]  $. For every
$n\in\omega$ we define $T_{n}=\left\{  s\in T\mid\left\vert s\right\vert
=n\right\}  .$ If $s\in\omega^{<\omega}$ then the \emph{cone of }$s$ is
defined as $\left\langle s\right\rangle =\left\{  f\in\omega^{\omega}\mid
s\subseteq f\right\}  .$ Letting $\mathcal{X}\subseteq\left[  \omega\right]
^{\omega},$ we say a tree $T\subseteq\omega^{<\omega}$ is a $\mathcal{X}%
$\emph{-branching tree} if \textsf{suc}$_{T}\left(  s\right)  \in\mathcal{X}$
for every $s\in T$.

\section{Preliminaries on ideals on countable sets}

In this section, we gather the definitions and results on ideals in countable
sets and the Kat\v{e}tov order that will be needed for the paper. Readers
interested in learning more about these topics are encouraged to consult
\cite{KatetovOrderonBorelIdeals}, \cite{OrderingMADFamiliesalaKatetov},
\cite{CombinatoricsofFiltersandIdeals}, \cite{StructuralKatetov},
\cite{ForcingIndestructibilityofMADFamilies}, \cite{KatetovOrderImply},
\cite{KatetovandKatetovBlassOrdersFsigmaIdeals}, \cite{SakaiKatetov},
\cite{KatetovMAD}, \cite{InvariancePropertiesofAlmostDisjointFamilies},
\cite{KwelaUltrafiltersKatetov}, \cite{KwelaExtendability} or
\cite{KatetovHindman} among others.

\qquad\ \ \ \qquad\ \ \ \ \ \ \ \ \ \ \ \ \ \ \ \ 

We will be mainly interested in filters and ideals on countable sets. Topology
turns out to be extremely useful when studying them. We endow $\mathcal{P}%
\left(  \omega\right)  $ with the natural topology that makes it homeomorphic
to $2^{\omega}$. In this way, the topology of $\mathcal{P}\left(
\omega\right)  $ has as a subbase the sets of the form $\left\langle
n\right\rangle _{0}=\left\{  A\subseteq\omega\mid n\notin A\right\}  $ and
$\left\langle n\right\rangle _{1}=\left\{  A\subseteq\omega\mid n\in
A\right\}  $, for $n\in\omega.$ We view filters and ideals as subspaces of
$\mathcal{P}\left(  \omega\right)  .$ All notions of Borel, analytic or meager
are referred to this topology. The following Borel ideals will play a key role
on the paper. For the next definitions, given $n\in\omega$, denote the column
$C_{n}=\left\{  \left(  n,m\right)  \mid m\in\omega\right\}  $ and
$\mathcal{C}=\left\{  C_{n}\mid n\in\omega\right\}  .$ Given $f\in
\omega^{\omega},$ denote $D\left(  f\right)  =\left\{  \left(  n,m\right)
\in\omega\times\omega\mid m\leq f\left(  n\right)  \right\}  .$

\begin{enumerate}
\item The ideal \textsf{fin }is the ideal of finite subsets of $\omega.$

\item The \emph{eventually different} ideal $\mathcal{ED}$ is the ideal on
$\omega^{2}$ generated by $\mathcal{C}$ and the graphs of functions from
$\omega$ to $\omega.$

\item The ideal $\mathcal{ED}_{\text{\textsf{fin}}}$ is the restriction of
$\mathcal{ED}$ to $\triangle=\left\{  \left(  n,m\right)  \mid m\leq
n\right\}  .$

\item The ideal \textsf{fin}$\times$\textsf{fin }is the ideal on $\omega^{2}$
generated by $\mathcal{C\cup}\left\{  D\left(  f\right)  \mid f\in
\omega^{\omega}\right\}  .$

\item The \emph{nowhere dense ideal, }\textsf{nwd }is the ideal of nowhere
dense subsets of the rational numbers.
\end{enumerate}

\qquad\ \ \qquad\ \ \ 

We now list some of the main notions concerning ideals on countable sets.

\begin{definition}
Let $\mathcal{I}$ be an ideal on $\omega$ (or any countable set).

\begin{enumerate}
\item $\mathcal{I}$ is \emph{tall }if for every $X\in\left[  \omega\right]
^{\omega}$ there is $Y\in\mathcal{I}$ such that $Y\cap X$ is infinite.

\item $\mathcal{I}$ is $\omega$\emph{-hitting }if for every $\left\{
X_{n}\mid n\in\omega\right\}  \subseteq\left[  \omega\right]  ^{\omega}$ there
is $Y\in\mathcal{I}$ such that $Y\cap X_{n}$ is infinite for every $n\in
\omega.$

\item $\mathcal{I}$ is \emph{tight }if for every $\left\{  X_{n}\mid
n\in\omega\right\}  \subseteq\mathcal{I}^{+},$ there is $Y\in\mathcal{I}$ such
that $Y\cap X_{n}$ is infinite for every $n\in\omega.$

\item $\mathcal{I}$ is a \emph{P-ideal }if every countable subfamily of
$\mathcal{I}$ has a pseudounion in $\mathcal{I}.$

\item $\mathcal{I}$ is a \emph{P}$^{+}$\emph{-ideal }if every $\subseteq
$-decreasing family $\left\{  X_{n}\mid n\in\omega\right\}  \subseteq
\mathcal{I}^{+}$ has a pseudointersection in $\mathcal{I}^{+}.$

\item $\mathcal{I}$ is a \emph{P}$^{-}$\emph{-ideal }if for every
$X\in\mathcal{I}^{+}$, any $\left\{  X_{n}\mid n\in\omega\right\}
\subseteq\left(  \mathcal{I}\upharpoonright X\right)  ^{\ast}$ has a
pseudointersection in $\mathcal{I}^{+}.$

\item $\mathcal{I}$ is a \emph{Q}$^{+}$\emph{-ideal }if for every
$X\in\mathcal{I}^{+}$ and every partition $\mathcal{P}=\left\{  P_{n}\mid
n\in\omega\right\}  $ of $X$ into finite sets, there is $A\in\mathcal{I}%
^{+}\cap\wp\left(  X\right)  $ such that $\left\vert A\cap P_{n}\right\vert
\leq1$ for every $n\in\omega.$ Such $A$ is called a \emph{partial selector of}
$\mathcal{P}$ and a \emph{selector }in case it intersects each $P_{n}.$

\item $\mathcal{I}$ is \emph{selective }if it is both \emph{P}$^{+}$ and
\emph{Q}$^{+}.$

\item $\mathcal{I}$ is \emph{weakly selective }if for every $X\in
\mathcal{I}^{+}$ and $\mathcal{P}$ a partition of $X,$ either $\mathcal{P\cap
I\neq\emptyset},$ or $\mathcal{P}$ has a (partial) selector in $\mathcal{I}%
^{+}.$

\item $\mathcal{I}$ is $+$\emph{-Ramsey }if every $\mathcal{I}^{+}$-branching
tree $T\,\ $has a branch in $\mathcal{I}^{+}.$

\item $\mathcal{I}$ is \emph{Cohen indestructible (}or $\mathbb{C}%
$-\emph{indestructible) }if $\mathcal{I}$ remains tall after adding a Cohen
real to the universe.
\end{enumerate}
\end{definition}

\qquad\ \ \ 

The notions dualize to filters. For example, a \emph{P-filter }is a filter
whose dual ideal is a \emph{P}-ideal. The same convention applies to all the
other properties listed above. We will switch between filters and ideals as
needed, depending on what is most convenient at the time. An ultrafilter is
called a \emph{Ramsey ultrafilter }if its dual is a selective ideal. The
following is a list of simple observations regarding the aforementioned notions:

\begin{enumerate}
\item $\omega$-hitting ideals are tall and tight.

\item A tall \textsf{P}-ideal is $\omega$-hitting.

\item An ideal is weakly selective if and only if it is \emph{P}$^{-}$ and
\emph{Q}$^{+}.$

\item Either \emph{P} or \emph{P}$^{+}$ implies \emph{P}$^{-}.$

\item A $+$-Ramsey ideal is weakly selective.

\item If $\mathcal{U}$ is an ultrafilter, then $\mathcal{U}^{\ast}$ is
$\omega$-hitting.

\item $\mathcal{I}$ is tight\emph{ }if and only if for every $\left\{
X_{n}\mid n\in\omega\right\}  \subseteq\mathcal{I}^{+},$ there is
$Y\in\mathcal{I}$ such that $Y\cap X_{n}\neq\emptyset$ for every $n\in\omega$
(in order to prove the non trivial implication, it is enough to realize we can
add to our original family all finite modifications of each $X_{n}$).
\end{enumerate}

\qquad\ \ \ \ \ \ \ \ 

We now look at cardinal invariants associated to ideals on countable sets.

\begin{definition}
Let $\mathcal{I}$ be a tall ideal on $\omega.$ Define:

\begin{enumerate}
\item \textsf{non*}$\left(  \mathcal{I}\right)  $ is the smallest size of a
family $\mathcal{H\subseteq}$ $\left[  \omega\right]  ^{\omega}$ such that for
every $A\in\mathcal{I},$ there is $H\in\mathcal{H}$ such that $A\cap H$ is finite.

\item \textsf{cof}$\left(  \mathcal{I}\right)  $ is the smallest size of a
cofinal family in $\left(  \mathcal{I},\subseteq\right)  .$
\end{enumerate}
\end{definition}

\qquad\ \ \ \ \ \ \ 

The invariant \textsf{non}$^{\ast}\left(  \mathcal{I}\right)  $ was introduced
by Brendle and Shelah in \cite{Ultrafiltersonomega}, but with a different
name. The notation used here follows
\cite{CardinalInvariantsofAnalyticPIdeals}. It is easy to see
that\textsf{non*}$\left(  \mathcal{I}\right)  $ $\leq$ \textsf{cof}$\left(
\mathcal{I}\right)  .$

\begin{definition}
Let $X,Y$ be two sets, $\mathcal{I}$ an ideal on $X$, $\mathcal{J}$ an ideal
on $Y$ and $f:X\longrightarrow Y.$

\begin{enumerate}
\item $f$ is a \emph{Kat\v{e}tov function from} $\mathcal{I}$ to $\mathcal{J}$
if for every $A\subseteq Y$, the following holds:

\hfill%
\begin{tabular}
[c]{l}%
If $A\in\mathcal{J},$ then $f^{-1}\left(  A\right)  \in\mathcal{I}.$%
\end{tabular}
\hfill\qquad

\item $f$ is a \emph{Kat\v{e}tov-Blass function from} $\mathcal{I}$ to
$\mathcal{J}$ if it is a Kat\v{e}tov function and it is finite to one.

\item $f$ is a \emph{Rudin-Keiser function from} $\mathcal{I}$ to
$\mathcal{J}$ if for every $A\subseteq Y$, the following holds:

\hfill%
\begin{tabular}
[c]{l}%
$A\in\mathcal{J}$ if and only if $f^{-1}\left(  A\right)  \in\mathcal{I}.$%
\end{tabular}
\hfill\qquad

\item $f$ is a \emph{Rudin-Blass function from} $\mathcal{I}$ to $\mathcal{J}$
if it is a Rudin-Keisler function and it is finite to one.

\item $\mathcal{J}\leq_{\text{\textsf{K}}}\mathcal{I}$ if there is a
Kat\v{e}tov function from $\mathcal{I}$ to $\mathcal{J}.$ The orders
$\leq_{\text{\textsf{KB}}},$ $\leq_{\text{\textsf{RK}}}$ and $\leq
_{\text{\textsf{RB}}}$ are defined analogously.

\item $\mathcal{I}$ and $\mathcal{J}$ are \emph{Kat\v{e}tov equivalent
}(denoted as $\mathcal{I=_{\text{\textsf{K}}}J}$) if $\mathcal{I\leq
_{\text{\textsf{K}}}J}$ and $\mathcal{J\leq_{\text{\textsf{K}}}I}.$
\end{enumerate}
\end{definition}

\qquad\ \ \ \qquad\ \ \qquad\ 

The following is a list of easy facts about the Kat\v{e}tov order.

\begin{lemma}
Let $\mathcal{I},$ $\mathcal{J}$ be ideals on $\omega$ and $X\subseteq\omega.$
\label{Lema Katetov basico}

\begin{enumerate}
\item If $\mathcal{I\subseteq J},$ then $\mathcal{I}\leq_{\text{\textsf{KB}}%
}\mathcal{J}.$

\item \textsf{fin }$\leq_{\text{\textsf{KB}}}\mathcal{I}.$

\item $\mathcal{I}$ is Kat\v{e}tov equivalent to \textsf{fin }if and only if
$\mathcal{I}$ is not tall.

\item If $X\in\mathcal{I}^{+},$ then $\mathcal{I}$ $\leq_{\text{\textsf{KB}}%
}\mathcal{I}\upharpoonright X.$

\item If $\mathcal{I}\leq_{\text{\textsf{KB}}}\mathcal{J},$ then
\textsf{non}$^{\ast}\left(  \mathcal{I}\right)  \leq$ \textsf{non}$^{\ast
}\left(  \mathcal{J}\right)  .$
\end{enumerate}
\end{lemma}

The following equivalence of the Kat\v{e}tov order is often useful:

\begin{lemma}
Let $X,Y$ be two sets, $\mathcal{I}$ an ideal on $X$, $\mathcal{J}$ an ideal
on $Y$ and $f:X\longrightarrow Y.$ The following are equivalent:
\label{EquivalenciaKatetov}

\begin{enumerate}
\item $f$ is a Kat\v{e}tov function from $\mathcal{I}$ to $\mathcal{J}.$

\item For every $A\subseteq X,$ if $A\in\mathcal{I}^{+},$ then $f\left[
A\right]  \in\mathcal{J}^{+}.$
\end{enumerate}
\end{lemma}

\qquad\qquad\qquad\qquad\ \ \ 

We have the following:

\begin{lemma}
\qquad\ \ \ \qquad\qquad\ \ \ 

\begin{enumerate}
\item $\mathcal{ED}$ $\leq_{\text{\textsf{KB}}}$ \textsf{fin}$\times
$\textsf{fin }and $\mathcal{ED}$ $\leq_{\text{\textsf{KB}}}\mathcal{ED}%
_{\text{\textsf{fin}}}.$

\item $\mathcal{ED}$ and \textsf{nwd }are Kat\v{e}tov incomparable.

\item \textsf{fin}$\times$\textsf{fin, }$\mathcal{ED}_{\text{\textsf{fin}}}$
and \textsf{nwd }are Kat\v{e}tov incomparable.
\end{enumerate}
\end{lemma}

There are simple combinatorial characterizations of the Kat\v{e}tov order when
one of the ideals is one of those we defined earlier.

\begin{center}%
\begin{tabular}
[c]{|c|c|c|}\hline
\textbf{Kat\v{e}tov relation} & \textbf{Equivalent to} & \textbf{Reference}%
\\\hline
&  & \\
$\mathcal{I}$ $\leq_{\text{\textsf{K}}}$ \textsf{fin} & $\mathcal{I}$ is not
tall & \cite{OrderingMADFamiliesalaKatetov}\\
&  & \\
$\mathcal{I\nleq}_{\text{\textsf{K}}}$ \textsf{nwd} & $\mathcal{I}$ is
$\mathbb{C}$ indestructible & \cite{OrderingMADFamiliesalaKatetov},
\cite{KurilicMAD}\\
&  & \\
$\mathcal{ED\nleq}_{\text{\textsf{K}}}$ $\mathcal{I}$ & Every partition of
$\omega$ into & \cite{CombinatoricsofFiltersandIdeals}\\
& sets in $\mathcal{I}$ has a selector in $\mathcal{I}^{+}$ & \\
&  & \\
$\mathcal{ED\nleq}_{\text{\textsf{K}}}$ $\mathcal{I}\upharpoonright X$ &
$\mathcal{I}$ is weakly selective & \cite{CombinatoricsofFiltersandIdeals}\\
for all $X\in\mathcal{I}^{+}$ &  & \\
&  & \\
$\mathcal{ED}_{\text{\textsf{fin}}}$ $\mathcal{\nleq}_{\text{\textsf{KB}}}$
$\mathcal{I}$ & Every partition of $\omega$ in finite &
\cite{RamseyTypePropertiesofIdeals}\\
& pieces has a selector in $\mathcal{I}^{+}$ & \\
&  & \\
$\mathcal{ED}_{\text{\textsf{fin}}}$ $\mathcal{\nleq}_{\text{\textsf{KB}}}$
$\mathcal{I}\upharpoonright X$ & $\mathcal{I}$ is \textsf{Q}$^{+}$ &
\cite{RamseyTypePropertiesofIdeals}\\
for all $X\in\mathcal{I}^{+}$ &  & \\
&  & \\
\textsf{fin}$\times$\textsf{fin }$\mathcal{\nleq}_{\text{\textsf{K}}}$
$\mathcal{I}$ & Every countable subfamily of $\mathcal{I}^{\ast}$ &
\cite{RamseyTypePropertiesofIdeals}\\
& has a pseudointersection in $\mathcal{I}^{+}$ & \\
&  & \\
\textsf{fin}$\times$\textsf{fin}$\mathcal{\ \mathcal{\nleq}_{\text{\textsf{K}%
}}I}\upharpoonright X$ & $\mathcal{I}$ is \textsf{P}$^{-}$ &
\cite{RamseyTypePropertiesofIdeals}\\
for all $X\in\mathcal{I}^{+}.$ &  & \\\hline
\end{tabular}
\smallskip

\textsf{Table 2. Combinatorial properties and the Kat\v{e}tov order}
\end{center}

\qquad\qquad\qquad\ \ \ 

The ideals $\mathcal{ED},$ \textsf{fin}$\times$\textsf{fin }and $\mathcal{ED}%
_{\text{\textsf{fin}}}$ are related in the following way:

\begin{proposition}
Let $\mathcal{I}$ be an ideal on $\omega$ such that $\mathcal{ED}$
$\leq_{\text{\textsf{K}}}$ $\mathcal{I}.$ Either \textsf{fin}$\times
$\textsf{fin }$\leq_{\text{\textsf{K}}}\mathcal{I}$ or there is $X\in
\mathcal{I}^{+}$ such that $\mathcal{ED}_{\text{\textsf{fin}}}$ $\leq
_{\text{\textsf{K}}}$ $\mathcal{I}\upharpoonright X.$
\end{proposition}

\begin{proof}
We use the equivalences from Table 2. Since $\mathcal{ED}$ $\leq
_{\text{\textsf{K}}}$ $\mathcal{I},$ there is a partition $\mathcal{P=}%
\left\{  P_{n}\mid n\in\omega\right\}  \subseteq\mathcal{I}$ such that every
selector of it is in $\mathcal{I}.$ We now wonder if the family $\left\{
\omega\setminus P_{n}\mid n\in\omega\right\}  \subseteq\mathcal{I}^{\ast}$ has
a pseudointersection in $\mathcal{I}^{+}.$ If not, we get that \textsf{fin}%
$\times$\textsf{fin }$\leq_{\text{\textsf{K}}}\mathcal{I}.$ In case there is
$X\in\mathcal{I}^{+}$ a pseudointersection, we conclude that $\mathcal{ED}%
_{\text{\textsf{fin}}}$ $\leq_{\text{\textsf{K}}}$ $\mathcal{I}\upharpoonright
X.$
\end{proof}

\qquad\qquad\qquad

In this way, the Category Dichotomy is equivalent to a trichotomy: Every Borel
ideal is either below \textsf{nwd}, or has a restriction above \textsf{fin}%
$\times$\textsf{fin }or has a restriction above $\mathcal{ED}%
_{\text{\textsf{fin}}}$. This is how the dichotomy was proved in
\cite{KatetovOrderonBorelIdeals}.

\qquad\qquad\qquad

We now look at some variants of $\mathcal{ED}_{\text{\textsf{fin}}}.$ We will
say an interval partition $P=\left\{  P_{n}\mid n\in\omega\right\}  $ is
\emph{increasing }if $\left\vert P_{n}\right\vert <\left\vert P_{n+1}%
\right\vert $ for every $n\in\omega.$

\begin{definition}
Let $P=\left\{  P_{n}\mid n\in\omega\right\}  $ be an increasing partition.
Define the ideal $\mathcal{ED}_{P}$ as the ideal generated by all selectors of
$P.$
\end{definition}

\qquad\ \ \ 

Note that $\mathcal{ED}_{\text{\textsf{fin}}}$ is of this form. It is easy to
see that all these ideals are Kat\v{e}tov-Blass equivalent; however even more
is true, as we will now see.

\begin{definition}
Let $\mathcal{I}$ and $\mathcal{J}$ be two ideals on $\omega.$ We say that
$\mathcal{I}$ \emph{and} $\mathcal{J}$ \emph{are isomorphic }if there is a
bijection $f\in\omega^{\omega}$ that is a Rudin-Keisler function from $\left(
\omega,\mathcal{I}\right)  $ to $\left(  \omega,\mathcal{J}\right)  .$
\end{definition}

\qquad\ \ \ 

The definition of isomorphism between ideals found in the book
\cite{IliasBook} is more general than the one presented here. However, for
tall ideals, this notion is equivalent to the one from the book, which is the
case of interest in this work.

\begin{proposition}
Let $P$ and $R$ be two increasing interval partitions. The ideals
$\mathcal{ED}_{P}$ and $\mathcal{ED}_{R}$ are isomorphic.\qquad\qquad
\end{proposition}

\begin{proof}
Write $P=\left\{  P_{n}\mid n\in\omega\right\}  $ and $R=\left\{  R_{n}\mid
n\in\omega\right\}  .$ For $s\in\left[  \omega\right]  ^{<\omega},$ denote
\textsf{cov}$_{P}\left(  s\right)  $ as the least number of intervals from $P$
that cover $s.$ The number \textsf{cov}$_{R}\left(  s\right)  $ is defined in
the same way. To prove the proposition, it is enough to find a function
$g\in\omega^{\omega}$ with the following properties:

\begin{enumerate}
\item $g$ is bijective.

\item For every $n\in\omega,$ we have that \textsf{cov}$_{R}\left(  g\left[
P_{n}\right]  \right)  \leq2.$

\item For every $m\in\omega,$ we have that \textsf{cov}$_{P}\left(
g^{-1}\left(  R_{m}\right)  \right)  \leq2.$
\end{enumerate}

\qquad\qquad\qquad\qquad\ \ 

Indeed, $g$ will be the desired isomorphism since the image of every selector
of $P$ will be covered by at most two selectors of $R$, and the other way
around. We will find $g$ by finite approximations. Define $\mathbb{P}$ as the
set of all $q$ such that for every $n\in\omega,$ the following conditions hold:

\begin{enumerate}
\item $g$ is a partial injective function from $\omega$ to $\omega$ with
finite domain.

\item If $P_{n}\subseteq$ \textsf{dom}$\left(  g\right)  ,$ then
\textsf{cov}$_{R}\left(  q\left[  P_{n}\right]  \right)  \leq2.$

\item If $R_{n}\subseteq$ \textsf{im}$\left(  g\right)  ,$ then \textsf{cov}%
$_{P}\left(  q^{-1}\left(  R_{n}\right)  \right)  \leq2.$

\item If $P_{n}$ $\nsubseteq$ \textsf{dom}$\left(  g\right)  ,$ then
\textsf{cov}$_{R}\left(  q\left[  P_{n}\right]  \right)  \leq1.$

\item If $R_{n}\nsubseteq$ \textsf{im}$\left(  g\right)  ,$ then
\textsf{cov}$_{P}\left(  q^{-1}\left(  R_{n}\right)  \right)  \leq1.$
\end{enumerate}

\qquad\qquad\qquad\ 

We order $\mathbb{P}$ by reverse inclusion. Letting $n\in\omega,$ define
$D_{n}=\{q\in\mathbb{P\mid}P_{n}\subseteq\mathsf{\ }$\textsf{dom}$\left(
q\right)  \}$ and $E_{n}=\{q\in\mathbb{P\mid}R_{n}\subseteq\mathsf{\ }%
$\textsf{im}$\left(  q\right)  \}.$ We claim that both sets are dense in
$\mathbb{P}.$ We verify it for $D_{n}.$ Pick $q\in\mathbb{P}$ such that
$P_{n}$ is not contained in the domain of $q.$ Choose $m$ large enough such
that $\left\vert P_{n}\right\vert \leq\left\vert R_{m}\right\vert $ and
$R_{m}\cap$ \textsf{im}$\left(  q\right)  =\emptyset.$ In case $P_{n}\cap$
\textsf{dom}$\left(  q\right)  =\emptyset,$ extend $q$ such that it maps
$P_{n}$ into $R_{m}.$ In the other case, extend it such that it sends
$P_{n}\setminus$ \textsf{dom}$\left(  q\right)  $ into $R_{m}.$

\qquad\qquad\qquad

Since we only have countably many dense sets, a trivial application of the
Rasiowa-Sikorski Lemma (see \cite{Kunen} or \cite{Jech}) yields the desired isomorphism.
\end{proof}

\qquad\ \ \ \ \ \ \ 

Meager ideals have very strong combinatorial properties, as we will now review.

\begin{definition}
Let $\mathcal{I}$ be an ideal on $\omega$ and $P=\left\{  P_{n}\mid n\in
\omega\right\}  $ a partition of $\omega$ into finite intervals. We say that
$P$ is a \emph{Talagrand partition for }$\mathcal{I}$ if for every
$X\subseteq\omega,$ if $X$ contains infinitely many elements of $P,$ then
$X\in\mathcal{I}^{+}.$
\end{definition}

\qquad\ \ \ \ \ 

The following is a classical theorem in the theory of ideals. The reader may
consult \cite{Barty} for a proof:

\begin{theorem}
[Talagrand, Jalali-Naini]Let $\mathcal{I}$ be an ideal on $\omega.$ The
following are equivalent: \label{Teorema Talagrand}

\begin{enumerate}
\item $\mathcal{I}$ is meager.

\item $\mathcal{I}$ has the Baire property.

\item $\mathcal{I}$ has a Talagrand partition.
\end{enumerate}
\end{theorem}

\qquad\qquad\ \ \ \ \qquad\ \ \ 

We have the following:

\begin{proposition}
Let $\mathcal{I}$ and $\mathcal{J}$ be ideals on $\omega.$ If $\mathcal{I}$ is
analytic and $\mathcal{J}\leq_{\text{\textsf{K}}}\mathcal{I},$ then
$\mathcal{J}$ is meager. \label{Prop preimagen analiticos}
\end{proposition}

\begin{proof}
Let $g\in\omega^{\omega}$ be a Kat\v{e}tov function from $\left(
\omega,\mathcal{I}\right)  $ to $\left(  \omega,\mathcal{J}\right)  .$ Define
$\mathcal{L=}\left\{  A\mid g^{-1}\left(  A\right)  \in\mathcal{I}\right\}  ,$
it follows that $\mathcal{J\subseteq L}$ and since $\mathcal{I}$ is analytic,
it is easy to see that $\mathcal{L}$ is analytic as well. In this way,
$\mathcal{L}$ has the Baire property (see \cite{Kechris}) and by Theorem
\ref{Teorema Talagrand}, it is meager. Since $\mathcal{J\subseteq L},$ we
conclude that $\mathcal{J}$ is also meager.
\end{proof}

\qquad\ \ \qquad\ \ \ \ \ \ \ 

It is worth pointing out that the Kat\v{e}tov predecesors of a meager ideal
are not necessarily meager. In fact, it can be shown that the meager ideals
are cofinal in the Kat\v{e}tov order, so the hypothesis that $\mathcal{I}$ is
analytic is essential.

\qquad\ \ \ \qquad\ \ \ \ \ 

The cardinal invariant \textsf{non*}$(\mathcal{ED}_{\text{\textsf{fin}}})$
will play an important role in this work. We will need the following theorem,
which was obtained by Meza, Minami and the fourth author.

\begin{theorem}
[H., Meza, Minami \cite{PairSplitting}]\label{cosas nonED}\qquad
\ \ \ \ \ \qquad\ \ \ \ \qquad\ \ \ \ 

\begin{enumerate}
\item \ \textsf{cov}$\left(  \mathcal{M}\right)  =$ \textsf{min}%
$\mathsf{\{}\mathfrak{d},$ \textsf{non}*$(\mathcal{ED}_{\text{\textsf{fin}}%
})\}.$

\item Let $\kappa$ be an infinite cardinal. The following are equivalent:

\begin{enumerate}
\item $\kappa<$ \textsf{non}*$(\mathcal{ED}_{\text{\textsf{fin}}}).$

\item For every increasing interval partition $P=\left\{  P_{n}\mid n\in
\omega\right\}  $ and $\mathcal{B}$ a family of size $\kappa$ consisting of
partial infinite selectors of $P$, there is a selector of $P$ that has
infinite intersection with every element of $\mathcal{B}.$
\end{enumerate}
\end{enumerate}
\end{theorem}

\section{Uniformity number of a class of ideals \label{seccion non}}

Letting $\Gamma$ a class of ideals on countable sets, we denote by
\textsf{non}$\left(  \Gamma\right)  $ the least cofinality of an ideal that is
not in $\Gamma.$ Of course, this definition only makes sense if there exists
an ideal that is not in $\Gamma$, which is always the case in the interesting
settings. The study of this type of cardinal invariants is quite useful, as it
allows us to deduce properties of an ideal \textquotedblleft for
free\textquotedblright\ (or more formally, by merely knowing its cofinality).
Although the notation used here is likely new, these invariants have been
studied for quite some time. Below, we summarize some already known results.

\begin{center}%
\begin{tabular}
[c]{|c|c|c|}\hline
\textbf{Cardinal \textbf{I}nvariant} & \textbf{Uniformity number of} &
\textbf{Reference}\\\hline
&  & \\
$\mathfrak{b}$ & Meager ideals & \cite{HandbookBlass}\\
&  & \\
$\mathfrak{d}$ & \textsf{P}$^{+}$-ideals & \cite{HandbookBlass}\\
&  & \\
\textsf{cov}$\left(  \mathcal{M}\right)  $ & $+$-Ramsey ideals &
\cite{SelectivityofAlmostDisjointFamilies}\\\hline
\end{tabular}
\smallskip

\textsf{Table 3. Some uniformity numbers of classes}
\end{center}

We are interested in the uniformity numbers for the classes of \textsf{P}%
$^{-},$ \textsf{Q}$^{+}$ and weakly selective ideals, which we denote by
\textsf{non}$($\textsf{P}$^{-}),$ \textsf{non}$($\textsf{Q}$^{+})$ and
\textsf{non}$($\textsf{WS}$)$ respectively. The case of for \textsf{non}%
$($\textsf{P}$^{-})$ is very easy:

\begin{proposition}
$\mathfrak{d=}$ \textsf{non}$($\textsf{P}$^{-}).$
\end{proposition}

\begin{proof}
Since \textsf{fin}$\times$\textsf{fin }is not \textsf{P}$^{-}$ and it has
cofinality $\mathfrak{d},$ we get that \textsf{non}$($\textsf{P}$^{-}%
)\leq\mathfrak{d}.$ On the other hand, as noted above, every ideal of
cofinality less than $\mathfrak{d}$ is \textsf{P}$^{+},$ so we conclude that
$\mathfrak{d}$ $\leq$ \textsf{non}$($\textsf{P}$^{-}).$
\end{proof}

\qquad\qquad\ \ \ \qquad\ \ 

For an ideal $\mathcal{I},$ denote the classes of ideals $K\left(
\mathcal{I}\right)  =\{\mathcal{J}\mid\mathcal{I}\nleq_{\text{\textsf{K}}%
}\mathcal{J}\}$ and $KB\left(  \mathcal{I}\right)  =\{\mathcal{J}%
\mid\mathcal{I}\nleq_{\text{\textsf{KB}}}\mathcal{J}\}.$ The following notion
was studied by Brendle and Fla\v{s}kov\'{a} \cite{BrendleFlaskova} \ and by
Hong and Zhang in \cite{HongZhang}.

\begin{definition}
Let $\mathcal{I}$ be an ideal on $\omega.$ The \emph{exterior cofinality}
(also called \emph{generic existence number}) of $\mathcal{I}$ is defined as
$\mathfrak{ge}\left(  \mathcal{I}\right)  =$ \textsf{min}$\mathsf{\{}%
$\textsf{cof}$\left(  \mathcal{J}\right)  \mid\mathcal{I\subseteq J}\}.$
\end{definition}

The generic existence number was introduced to investigate the \emph{generic
existence} of certain special classes of ultrafilters. Specifically, it refers
to the property that any filter with cofinality less than $\mathfrak{c}$ can
be extended to an ultrafilter within that class. For further details, we refer
the interested reader to the previously mentioned papers.

\begin{lemma}
Let $\mathcal{I}$ be an ideal on $\omega.$ The following cardinal invariants
are equal:

\begin{enumerate}
\item $\mathfrak{ge}\left(  \mathcal{I}\right)  .$

\item \textsf{min}$\mathsf{\{}$\textsf{cof}$\left(  \mathcal{J}\right)
\mid\mathcal{I\leq_{\text{\textsf{K}}}J}\}.$

\item \textsf{min}$\mathsf{\{}$\textsf{cof}$\left(  \mathcal{J}\right)
\mid\mathcal{I\leq_{\text{\textsf{KB}}}J}\}.$

\item \textsf{non}$\left(  K\left(  \mathcal{I}\right)  \right)  .$

\item \textsf{non}$\left(  KB\left(  \mathcal{I}\right)  \right)  .$
\end{enumerate}
\end{lemma}

\begin{proof}
The proof of the equality between 1,2 and 3 can be found implicitly in
Observation 3.1 of \cite{BrendleFlaskova}. Finally, \textsf{non}$\left(
K\left(  \mathcal{I}\right)  \right)  $ is obviously equivalent to the
invariant in point 2 and \textsf{non}$\left(  KB\left(  \mathcal{I}\right)
\right)  $ is clearly equivalent to the one in point 3.
\end{proof}

\qquad\qquad\ \ \ \ 

We conclude that \textsf{non}$($\textsf{Q}$^{+})=$ $\mathfrak{ge(}%
\mathcal{ED}_{\text{\textsf{fin}}})$ and \textsf{non}$($\textsf{WS}%
$)=\mathfrak{ge(}\mathcal{ED})$ (here we are using the fact that
\textsf{cof}$\left(  \mathcal{I}\upharpoonright X\right)  \leq$ \textsf{cof}%
$\left(  \mathcal{I}\right)  $ for an ideal $\mathcal{I}$ and $X\in
\mathcal{I}^{+}$). In \cite{BrendleFlaskova} several generic existence numbers
are computed. In particular, we can find the following:

\begin{proposition}
[Brendle, Fla\v{s}kov\'{a} \cite{BrendleFlaskova}]The following holds:
\label{non weakly selective}

\begin{enumerate}
\item \textsf{non}$($\textsf{WS}$)=$ \textsf{cov}$\left(  \mathcal{M}\right)
.$

\item \textsf{non}*$(\mathcal{ED}_{\text{\textsf{fin}}})\leq$ \textsf{non}%
$($\textsf{Q}$^{+}).$
\end{enumerate}
\end{proposition}

The following was conjecture in \cite{BrendleFlaskova}:

\begin{conjecture}
[Brendle, Fla\v{s}kov\'{a}]The cardinals \textsf{non}*$(\mathcal{ED}%
_{\text{\textsf{fin}}})$ and \textsf{non}$($\textsf{Q}$^{+})$ are equal.
\end{conjecture}

\qquad\ \ \ \ \ 

The conjecture still remains unsolved. These two cardinal invariants certain
look very similar, for example, we have the following:

\begin{corollary}
\textsf{cov}$\left(  \mathcal{M}\right)  =$ \textsf{min}$\mathsf{\{}%
\mathfrak{d},$ \textsf{non}*$(\mathcal{ED}_{\text{\textsf{fin}}})\}=$
\textsf{min}$\mathsf{\{}\mathfrak{d},$ \textsf{non}$($\textsf{Q}$^{+})\}.$
\end{corollary}

\begin{proof}
We already know \textsf{cov}$\left(  \mathcal{M}\right)  =$ \textsf{min}%
$\mathsf{\{}\mathfrak{d},$ \textsf{non}*$(\mathcal{ED}_{\text{\textsf{fin}}%
})\}$ by Theorem \ref{cosas nonED}. To see the other equality, just note that
\textsf{non}$($\textsf{WS}$)$ is the minimum between \textsf{non}$($%
\textsf{Q}$^{+})$ and \textsf{non}$($\textsf{P}$^{-}).$\qquad\ \ \ \ 
\end{proof}

\section{The Category Dichotomy for Maximal ideals \label{seccion ultrafiltro}%
}

It is very easy to characterize when the Category Dichotomy holds for maximal
ideals (duals of ultrafilters).

\begin{proposition}
The following are equivalent:

\begin{enumerate}
\item The Category Dichotomy for the class of maximal ideals.

\item There are no Ramsey ultrafilters.
\end{enumerate}
\end{proposition}

\begin{proof}
An ultrafilter $\mathcal{U}$ is Ramsey if and only if $\mathcal{U}^{\ast}$ is
not Kat\v{e}tov above $\mathcal{ED}$. In this way, in order to prove the
proposition, it is enough to show that no maximal ideal can be Kat\v{e}tov
below \textsf{nwd. }In fact, by Proposition \ref{Prop preimagen analiticos},
the dual of an ultrafilter is not Kat\v{e}tov below any analytic ideal
(alternatively, we could use Lemma \ref{Lema tight no abajo} below).

\qquad\qquad\qquad\ \ \ \ \qquad\qquad\qquad
\end{proof}

It is very easy to build Ramsey ultrafilters with \textsf{CH }and a classic
theorem of Kunen is that it is consistent that they do not exist (see
\cite{KunenSomePoints}). In this way, the Category Dichotomy for the class of
duals of ultrafilters is independent. The reader may consult
\cite{ProperandImproper}, \cite{Rat}, \cite{NoQPOintsLaverModel},
\cite{ShelahNowheredense}, \cite{Brendlenowheredense},
\cite{TherearenoPpointsinSilverExtensions} and \cite{AboveFsigma} to learn
more about models where certain ultrafilters do not exist.

\section{The Category Dichotomy for MAD families \label{seccion MAD}}

We say $\mathcal{A}\subseteq\left[  \omega\right]  ^{\omega}$ is \emph{almost
disjoint (\textsf{AD}) }if the intersection of any two different elements of
$\mathcal{A}$ is finite. A \textsf{MAD }\emph{family}\textsf{\emph{ }}is
a\textsf{ }maximal almost disjoint family. We denote by $\mathcal{I(A)}$ the
ideal generated by $\mathcal{A}.$ A \textsf{MAD }family is \emph{Cohen
indestructible }if it remains maximal after adding a Cohen real, which is
equivalent that its ideal is not Kat\v{e}tov below \textsf{nwd. }We now prove
the following:

\begin{proposition}
The following are equivalent:

\begin{enumerate}
\item The Category Dichotomy for the class of ideals generated by \textsf{MAD
}families.

\item There are no Cohen indestructible \textsf{MAD }families.
\end{enumerate}
\end{proposition}

\begin{proof}
It is enough to prove that if $\mathcal{A}$ is a \textsf{MAD }family, then no
restriction of $\mathcal{I(A)}$ is Kat\v{e}tov above $\mathcal{ED}.$ This is
true because $\mathcal{I(A)}$ is selective\textsf{ }see \cite{Happyfamilies},
\cite{RamseySpaces} or the survey \cite{AlmostDisjointFamiliesandTopology}.
Even more is true, no restriction of $\mathcal{I(A)}$ can be Kat\v{e}tov above
a tall analytic ideal. This is because of a theorem of Mathias that an
ultrafilter is Ramsey if and only if it intersects every tall analytic ideal.
\end{proof}

\qquad\qquad\qquad\qquad

While it is consistent that Cohen indestructible \textsf{MAD }families exist,
it is unknown if it is consistent that they do not (this is a famous problem
of Stepr\={a}ns). In this way, the Category Dichotomy for the \textsf{MAD
}families is consistently false, but we do not know if it may consistently be
true. The reader may consult \cite{OrderingMADFamiliesalaKatetov},
\cite{KurilicMAD}, \cite{ForcingIndestructibilityofMADFamilies},
\cite{ForcingwithQuotients} and \cite{SurveyDestructibility} to learn more
about indestructibility of \textsf{MAD }families and ideals.

\section{The Category Dichotomy for P-ideals\label{seccion P ideal}}

The following result is known, we prove it for the sake of completeness:
\label{Lema tight no abajo}

\begin{lemma}
Let $\mathcal{I}$ be an ideal on $\omega.$

\begin{enumerate}
\item If $\mathcal{I}$ is tight, then $\mathcal{I}$ $\nleq_{\text{\textsf{K}}%
}$ \textsf{nwd.}

\item If $\mathcal{I}$ is a \textsf{P}-ideal, then \textsf{fin}$\times
$\textsf{fin }$\nleq_{\text{\textsf{K}}}\mathcal{I}.$
\end{enumerate}
\end{lemma}

\begin{proof}
First assume $\mathcal{I}$ is tight and $f:\mathbb{Q}\longrightarrow\omega.$
We will prove that $f$ is not a Kat\v{e}tov from ($\mathbb{Q},$\textsf{nwd)
}to\textsf{ }$\left(  \omega,\mathcal{I}\right)  .$ If there is an open non
empty set of $\mathbb{Q}$ whose image is in $\mathcal{I},$ there is nothing to
do (see Lemma \ref{EquivalenciaKatetov}), so assume this is not the case. Let
$\left\{  U_{n}\mid n\in\omega\right\}  $ be a base of open sets of
$\mathbb{Q}.$ Since each $f\left[  U_{n}\right]  \in\mathcal{I}^{+},$ we can
find $A\in\mathcal{I}$ having infinite intersection with all of them. In this
way $f^{-1}\left(  A\right)  $ is dense, so it is not nowhere dense.

\qquad\qquad\qquad

Now assume $\mathcal{I}$ is a \textsf{P}-ideal and $f:\omega\longrightarrow
\omega^{2},$ We will prove $f$ is not a Kat\v{e}tov function from $\left(
\omega,\mathcal{I}\right)  $ to ($\omega^{2},$\textsf{fin}$\times
$\textsf{fin). }For $n\in\omega,$ denote the column $C_{n}=\left\{  n\right\}
\times\omega.$ If there is $n\in\omega$ such that $f^{-1}\left(  C_{n}\right)
\notin\mathcal{I},$ there is nothing to do. In the other case, since
$\mathcal{I}$ is a \textsf{P}-ideal, we can find $A\in\mathcal{I}$ such that
$f^{-1}\left(  C_{n}\right)  \subseteq^{\ast}A$ for every $n\in\omega.$ In
this way, the image of the complement of $A$ is in \textsf{fin}$\times
$\textsf{fin, }hence $f$ is not Kat\v{e}tov.
\end{proof}

\qquad\qquad\qquad\ \ \ \ 

In this way, the Category Dichotomy for the class of \textsf{P}-ideals is
equivalent to the statement that every \textsf{P}-ideal has a restriction
above $\mathcal{ED}_{\text{\textsf{fin}}}$ \ (equivalently, they are not
\textsf{Q}$^{+}$). Note that item 2 in the lemma above is false for $\omega
$-hitting ideals. For example, take a maximal ideal extending \textsf{fin}%
$\times$\textsf{fin.}

\begin{definition}
Let $\mathcal{T}$ be an increasing tower. Define $\mathcal{I}\left(
\mathcal{T}\right)  $ as the ideal generated by $\mathcal{T}.$ For
$\mathcal{T}$ a tower, denote $\mathcal{F}\left(  \mathcal{T}\right)  $ the
filter it generates.
\end{definition}

\qquad\qquad\ \ 

Since the length of a (increasing) tower can not have countable cofinality, it
follows that its ideal (filter) is a \textsf{P}-ideal (\textsf{P}-filter). We
can now prove the following:

\begin{theorem}
Both $\mathfrak{t<}$ \textsf{cov}$\left(  \mathcal{M}\right)  $ and
$\mathfrak{t}=\mathfrak{d}$ imply that there is a \textsf{P}-ideal that is
\textsf{Q}$^{+}$. In this way, the Category Dichotomy for the class of
\textsf{P}-ideals fails.
\end{theorem}

\begin{proof}
If $\mathfrak{t<}$ \textsf{cov}$\left(  \mathcal{M}\right)  ,$ then by Theorem
\ref{non weakly selective}, the ideal generated by an increasing tower is
weakly selective and we are done. Now we assume $\mathfrak{t}=\mathfrak{d.}$
Let $\mathcal{P}$ $=\left\{  P_{\alpha}\mid\alpha<\mathfrak{d}\right\}  $ be a
dominating family of interval partitions. We recursively define $\mathcal{B}%
=\left\{  B_{\alpha}\mid\alpha<\mathfrak{d}\right\}  $ such that for every
$\alpha<\mathfrak{d}$:

\begin{enumerate}
\item $\mathcal{B}$ is a tower.

\item $B_{\alpha}$ is a partial selector of $P_{\alpha}.$ Moreover, either
$B_{\alpha}$ is contained in the union the even intervals or $P_{\alpha}$ or
in the union of the odd intervals.\qquad
\end{enumerate}

\qquad\ \ \ \qquad
\ \ \ \ \ \ \ \ \ \ \ \ \ \ \ \ \ \ \ \ \ \ \ \ \ \ \ \ \ \ \ \ \ \ \ \ \ \ \ \qquad
\qquad\ \ 

This is easy to do (using that $\mathfrak{t}=\mathfrak{d}$). Note that item 2
above implies that $\mathcal{B}$ is actually a tower. We claim that
$\mathcal{F}\left(  \mathcal{B}\right)  $ is as desired. First, we claim that
for every interval partition $R,$ there is $B\in\mathcal{F}\left(
\mathcal{B}\right)  $ that is a partial selector of $R.$ Since $\mathcal{P}$
is a dominating family, we can find $\alpha<\mathfrak{d}$ such that every
interval of $P_{\alpha}$ contains one of $R.$ This implies that every interval
of $R$ can intersect at most two intervals of $P_{\alpha}.$ Since $B_{\alpha}$
is a partial selector of $P_{\alpha}$ and it never intersects two consecutive
elements of $P_{\alpha},$ it follows that $B_{\alpha}$ is a partial selector
of $R.$ Finally, we argue that $\mathcal{F}\left(  \mathcal{B}\right)  $ is a
\textsf{Q}$^{+}$-filter. Let $X\in\mathcal{F}\left(  \mathcal{B}\right)  ^{+}$
and $Q$ a partition of $X$ into finite pieces. Define $R=Q\cup\left\{
\left\{  n\right\}  \mid n\in\omega\setminus X\right\}  ,$ which is a
partition of $\omega$ into finite pieces. In this way, there is $B\in
\mathcal{F}\left(  \mathcal{B}\right)  $ a partial selector of $R.$ It is
clear that $B\cap X\in\mathcal{F}\left(  \mathcal{B}\right)  ^{+}$ and is a
partial selector of $Q.$
\end{proof}

\qquad\qquad

So the Category Dichotomy of \textsf{P}-ideals is consistently false,
surprisingly, it is also consistently true, in fact, we will now prove that it
holds in the Laver model. We first review some notions regarding Laver
forcing. Let $T\subseteq\omega^{<\omega}$ be a tree and $s\in T.$ We say $s$
\emph{is the stem of }$T$ (denoted as $s=$ \textsf{st}$\left(  T\right)  $) if
every node of $T$ is comparable with $s$ and it is maximal with this property.
A tree $p\subseteq\omega^{<\omega}$ is a \emph{Laver tree }if it consists of
increasing sequences, it has a stem and if $t\in T$ extends the stem of $T,$
then \textsf{suc}$_{T}\left(  s\right)  $ is infinite. The set of Laver trees
is denoted by $\mathbb{L}$ and given $p,q\in\mathbb{L},$ denote $p\leq q$ if
$p\subseteq q.$ Moreover, define $p\leq_{0}q$ if $p\leq q$ and \textsf{st}%
$\left(  p\right)  =$ \textsf{st}$\left(  q\right)  .$ If $p\in\mathbb{L}$ and
$s\in p,$ define $p_{s}=\left\{  t\in p\mid t\subseteq s\vee s\subseteq
t\right\}  .$ It is clear that $p_{s}\in\mathbb{L},$ it extends $p$ and in
case \textsf{st}$\left(  p\right)  \subseteq s,$ we have that \textsf{st}%
$\left(  p_{s}\right)  =s.$

\qquad\qquad\ \ \ \ 

The \emph{Laver generic real }will be denoted by $l_{gen}\in\omega^{\omega}$
as is the only element that is a branch of every tree in the generic filter.
The \emph{generic Laver interval partition}\ is defined as $P_{gen}=\left\{
P_{gen}\left(  n\right)  \mid n\in\omega\right\}  $ where $P_{gen}\left(
n\right)  =$ $[l_{gen}\left(  n\right)  ,l_{gen}\left(  n+1\right)  ).$ This
is an interval partition of almost all $\omega$ (we reiterate that for us,
Laver trees consist of increasing functions).

\begin{theorem}
[See \cite{Barty}]\emph{(Pure decision property) }Let $p\in\mathbb{L},$ $X$ a
finite set and $\dot{a}$ a $\mathbb{L}$-name such that $p\Vdash$%
\textquotedblleft$\dot{a}\in X$\textquotedblright. There is $q\leq_{0}p$ and
$b\in X$ such that $q\Vdash$\textquotedblleft$\dot{a}=b$\textquotedblright. \ 
\end{theorem}

\qquad\qquad\qquad\ \ \ \qquad\qquad

We will need the following lemma:

\begin{lemma}
Let $p\in\mathbb{L}$ with $s=$ \textsf{st}$\left(  p\right)  \in\omega^{n}$
and $\dot{a}$ such that $p\Vdash$\textquotedblleft$\dot{a}\in\dot{P}%
_{gen}\left(  n\right)  $\textquotedblright. There is $q\leq_{0}p$ such that
one of the following holds:

\begin{enumerate}
\item There is a function $g:$ \textsf{suc}$_{q}\left(  s\right)
\longrightarrow\omega$ injective such that $q_{s^{\frown}i}\Vdash
$\textquotedblleft$\dot{a}=g\left(  i\right)  $\textquotedblright\ for every
$i\in$ \textsf{suc}$_{q}\left(  s\right)  .$

\item There is a sequence $\langle h_{i}\mid i\in$ \textsf{suc}$_{q}\left(
s\right)  \rangle$ such that for every $i\in$ \textsf{suc}$_{q}\left(
s\right)  ,$ we have that $h_{i}$ :\textsf{suc}$_{q}\left(  s\right)
\longrightarrow\omega$ is injective and $q_{s^{\frown}i^{\frown}j}\Vdash
$\textquotedblleft$\dot{a}=h_{i}\left(  j\right)  $\textquotedblright\ for
every $j\in$ \textsf{suc}$_{q}\left(  s^{\frown}i\right)  .$
\end{enumerate}
\end{lemma}

\begin{proof}
Note that since $s\in\omega^{n},$ then $p$ has not decided $l_{gen}\left(
n\right)  $ or $l_{gen}\left(  n+1\right)  .$ Moreover, if $s^{\frown
}i^{\frown}j\in p,$ then $p_{s^{\frown}i^{\frown}j}\Vdash$\textquotedblleft%
$\dot{P}_{gen}\left(  n\right)  =[i,j)$\textquotedblright. Now, for every
$t\in p$ with $\left\vert t\right\vert =n+2,$ we can apply the pure decision
property to $p_{t}$ and $\dot{a}.$ In this way, we find $r\leq_{0}p$ such that
for every $t\in r$ with $\left\vert t\right\vert =n+2,$ there is $a_{t}$ such
that $r_{t}\Vdash$\textquotedblleft$\dot{a}=a_{t}$\textquotedblright. We can
now find $\overline{r}\leq_{0}r$ such that for every $i\in$ \textsf{suc}%
$_{\overline{r}}\left(  s\right)  ,$ one of the following holds:

\begin{enumerate}
\item $a_{t}=a_{z}$ for every $t,z\in\overline{r}$ such that $\left\vert
t\right\vert =\left\vert z\right\vert =n+2$ and $t\left(  n\right)  =z\left(
n\right)  =i.$

\item $a_{t}\neq a_{z}$ for every $t,z\in\overline{r}$ distinct such that
$\left\vert t\right\vert =\left\vert z\right\vert =n+2$ and $t\left(
n\right)  =z\left(  n\right)  =i.$
\end{enumerate}

\qquad\ \ 

If there are infinitely many $i\in$ \textsf{suc}$_{\overline{r}}\left(
s\right)  $ for which condition 2 holds, we can easily find and extension of
$\overline{r}$ satisfying item 2 of the lemma. Assume that condition 1 holds
for almost all $i\in$ \textsf{suc}$_{\overline{r}}\left(  s\right)  .$ By
pruning $\overline{r},$ we can assume that in fact this holds for all. For
every $i\in$ \textsf{suc}$_{\overline{r}}\left(  s\right)  ,$ let $a_{i}$ be
the common value obtained in condition 1 above. We claim that for every
$k\in\omega,$ there are only finitely many $i\in$ \textsf{suc}$_{\overline{r}%
}\left(  s\right)  $ such that $a_{i}=k.$ This is because if $i>k,$ then
$p_{s^{\frown}i}\Vdash$\textquotedblleft$k<i=l_{gen}\left(  n\right)  \leq
\dot{a}$\textquotedblright. We can now easily obtain an extension of
$\overline{r}$ for which item 1 in the lemma holds.
\end{proof}

\qquad\qquad\qquad\ \ \ \ \ \ \qquad\qquad\qquad\ \ 

With the lemma at our disposal, we now prove the following:

\begin{proposition}
Let $\mathcal{I}$ be a tall \textsf{P}-ideal, $p\in\mathbb{L}$ and $\dot{X}$
be a $\mathbb{L}$-name for a selector of $P_{gen}.$ There is $q\leq_{0}p$ and
$B\in\mathcal{I}$ such that $q\Vdash$\textquotedblleft$\dot{X}\subseteq
B$\textquotedblright. \label{Laver 1 paso}
\end{proposition}

\begin{proof}
By applying the previous lemma infinitely many times, without lost of
generality, we may assume that for every $t\in p$ below the stem with
$\left\vert t\right\vert =n,$ one of the following holds:

\begin{enumerate}
\item There is a function $g^{t}:$ \textsf{suc}$_{p}\left(  t\right)
\longrightarrow\omega$ injective such that $p_{t^{\frown}i}\Vdash
$\textquotedblleft$\dot{a}=g^{t}\left(  i\right)  $\textquotedblright\ for
every $i\in$ \textsf{suc}$_{p}\left(  t\right)  .$

\item There is a sequence $\langle h_{i}^{t}\mid i\in$ \textsf{suc}%
$_{p}\left(  t\right)  \rangle$ such that for every $i\in$ \textsf{suc}%
$_{p}\left(  t\right)  ,$ we have that $h_{i}^{t}:$ \textsf{suc}$_{p}\left(
t^{\frown}i\right)  \longrightarrow\omega$ is injective and $p_{t^{\frown
}i^{\frown}j}\Vdash$\textquotedblleft$\dot{a}=h_{i}^{t}\left(  j\right)
$\textquotedblright\ for every $j\in$ \textsf{suc}$_{p}\left(  t^{\frown
}i\right)  .$
\end{enumerate}

\qquad\qquad\ \ \qquad\ \ 

We will say $t\in p$ (below the stem) is of \emph{Type 1} if condition 1 above
holds and is of \emph{Type 2} if condition 2 above is the one that is true.
Let $t\in p$ below the stem. If $t$ is of Type 1, denote $W\left(  t\right)
=\left\{  im\left(  g^{t}\right)  \right\}  $ and if it is of Type 2, denote
$W\left(  t\right)  =\{im\left(  h_{i}^{t}\right)  \mid i\in$ \textsf{suc}%
$_{p}\left(  t\right)  \}.$ Since $\mathcal{I}$ is a tall \textsf{P}-ideal, we
can find a single $B\in\mathcal{I}$ such that $B$ has infinite intersection
with every element of each $W\left(  t\right)  $ for all $t\in p$ below the
stem. Recursively, we can now find $q\leq_{0}p$ such that for every $t\in q$
below the stem, we have the following:

\begin{enumerate}
\item If $t$ is of Type 1, then the image of $g^{t}\upharpoonright
$\textsf{suc}$_{q}\left(  t\right)  $ is contained in $B.$

\item If $t$ is of Type 2, then for every $i\in$ \textsf{suc}$_{q}\left(
t\right)  ,$ we have that the image of $h_{i}^{t}\upharpoonright$%
\textsf{suc}$_{q}\left(  t^{\frown}i\right)  $ is contained in $B.$
\end{enumerate}

\qquad\qquad\ \ \ \ \ 

In case the stem of $p$ is empty, we get that $q\Vdash$\textquotedblleft%
$\dot{X}\subseteq B$\textquotedblright\ and we are done. If the stem of $p$
has size $m+1,$ we get that $q\Vdash$\textquotedblleft$\dot{X}\setminus\dot
{l}_{gen}\left(  m\right)  \subseteq B$\textquotedblright, so $q\Vdash
$\textquotedblleft$\dot{X}\subseteq B\cup\dot{l}_{gen}\left(  m\right)
$\textquotedblright\ and we are done.
\end{proof}

\qquad\qquad\qquad\ \ \ \qquad\ \ \ \qquad\ \ \ \ \ \ 

To handle the iteration, we need a slightly stronger result. The proof of the
following proposition is essentially the same as that of the previous one,
with only a slight increase in notational complexity.

\begin{proposition}
Let $\mathcal{I}$ be a tall \textsf{P}-ideal, $p\in\mathbb{L}$ and $\dot{H}$
be a $\mathbb{L}$-name for a function such that $p\Vdash$\textquotedblleft%
$\forall n\in\omega(\dot{H}\left(  n\right)  \in\left[  P_{gen}\left(
n\right)  \right]  ^{\leq n+1})$\textquotedblright$.$ There is $q\leq_{0}p$
and $B\in\mathcal{I}$ such that $q\Vdash$\textquotedblleft$%
{\textstyle\bigcup\limits_{n\in\omega}}
\dot{H}\left(  n\right)  \subseteq B$\textquotedblright.
\label{no rapid selectors}
\end{proposition}

We now recall the following notion:

\begin{definition}
Let $\mathbb{P}$ be a partial order. We say $\mathbb{P}$ has the \emph{Laver
property }if for every $p\in\mathbb{P}$ and $\dot{f}$ a $\mathbb{P}$-name for
a bounded function, there are $q\leq p$ and $S:\omega\longrightarrow\left[
\omega\right]  ^{<\omega}$ such that or every $n\in\omega$ we have that
$\left\vert S\left(  n\right)  \right\vert \leq n+1$ and $q\Vdash
$\textquotedblleft$\dot{f}\left(  n\right)  \in S\left(  n\right)
$\textquotedblright.
\end{definition}

\qquad\qquad\qquad\qquad\qquad\qquad\qquad

It is well-known that Laver forcing, as well as its iterations, have the Laver
property (see \cite{Barty} or \cite{ProperandImproper}). With Proposition
\ref{no rapid selectors}, we conclude the following:

\begin{proposition}
Let $\mathcal{I}$ be a tall \textsf{P}-ideal and $\mathbb{\dot{P}}$ be an
$\mathbb{L}$-name for a forcing notion with the Laver property. $\mathbb{L\ast
\dot{P}}$ forces that every selector of $P_{gen}$ is in $\mathcal{I}.$
\label{laver prop selectores chicos}
\end{proposition}

\qquad\qquad\qquad\ \ \ 

By the \emph{Laver model, }we mean a model obtained by iterating Laver forcing
with countable support $\omega_{2}$ many times over a model of \textsf{CH}. We
can now prove the following:

\begin{theorem}
In the Laver model, every tall \textsf{P}-ideal is Kat\v{e}tov above
$\mathcal{ED}_{\text{\textsf{fin}}}.$ Hence, the Category Dichotomy for tall
\textsf{P}-ideals holds. \label{In the Laver model}
\end{theorem}

\begin{proof}
Follows by Proposition \ref{laver prop selectores chicos} and the fact that
for every tall \textsf{P}-ideal in the final extension, we can find an
intermediate extension in which it is also a tall \textsf{P}-ideal.
\end{proof}

\qquad\ \ \ \ \ \ \ \ 

We do not know if we can extend Theorem \ref{In the Laver model} for $\omega
$-hitting ideals. While Proposition \ref{Laver 1 paso} is also true for them,
it is not clear if Proposition \ref{no rapid selectors} holds as well. We ask
the following:

\begin{problem}
\qquad\ \ \ 

\begin{enumerate}
\item Is it true that in the Laver model, every $\omega$-hitting ideal is
Kat\v{e}tov above $\mathcal{ED}_{\text{\textsf{fin}}}$?

\item Is it consistent that the Category Dichotomy holds for the $\omega
$-hitting ideals?
\end{enumerate}
\end{problem}

\section{Every non meager ideal may be above $\mathcal{ED}_{\text{\textsf{fin}%
}}$ \label{no magro y Q mas}}

In this short section, we will prove that it is consistent that every non
meager ideal is Kat\v{e}tov-Blass above $\mathcal{ED}_{\text{\textsf{fin}}}$
(in particular, it is not \textsf{Q}$^{+}$). Note that by Proposition
\ref{Prop preimagen analiticos}, we have the following:

\begin{corollary}
The following are equivalent:

\begin{enumerate}
\item The Category Dichotomy for non-meager ideals.

\item Every non-meager ideal has a restriction which is Kat\v{e}tov above
$\mathcal{ED}.$ 
\end{enumerate}
\end{corollary}

\qquad\qquad\qquad\qquad\ \ \ \ \ \qquad\qquad

The dual of a Ramsey ultrafilter is a counterexample for the Category
Dichotomy for non-meager ideals. Interestingly, the dichotomy for this class
is consistent. We need to recall the following important principle:

\begin{center}%
\begin{tabular}
[c]{c}%
\textbf{Filter Dichotomy}\\
\\
Let $\mathcal{I}$ be an ideal on $\omega.$ Either $\mathcal{I}$ is meager of
there is\\
an ultrafilter $\mathcal{U}$ such that $\mathcal{U}^{\ast}$ $\leq
_{\text{\textsf{RB}}}\mathcal{I}.$%
\end{tabular}

\qquad\ \ \ \ \ \qquad\ \ \ 
\end{center}

In \cite{InequivalentGrowthTypes} Blass and Laflamme proved that the Filter
Dichotomy holds in the Miller and Blass-Shelah models. We need also to mention
that Laflamme Zhou showed in \cite{RudinBlassOrderingofUltrafilters} that the
Filter Dichotomy implies that there are no \textsf{Q}-points. The main result
of this section easily follows:

\begin{theorem}
The Filter Dichotomy implies that every non meager ideal is Kat\v{e}tov-Blass
above $\mathcal{ED}_{\text{\textsf{fin}}}.$ In particular, it implies the
Category Dichotomy for non-meager ideals.
\end{theorem}

\begin{proof}
Let $\mathcal{I}$ be a non meager ideal. By the Filter Dichotomy, there is
$\mathcal{U}$ an ultrafilter such that $\mathcal{U}^{\mathcal{\ast}}$
$\leq_{\text{\textsf{RB}}}\mathcal{I}.$ By the theorem of Laflamme and Zhou
mentioned above, we know that $\mathcal{U}$ is not a \textsf{Q}-point. In this
way, $\mathcal{ED}_{\text{\textsf{fin}}}\leq_{\text{\textsf{KB}}}%
\mathcal{U}^{\ast}\leq_{\text{\textsf{RB}}}\mathcal{I},$ hence $\mathcal{ED}%
_{\text{\textsf{fin}}}\leq_{\text{\textsf{KB}}}\mathcal{I}.$
\end{proof}

\qquad\ \ \ \qquad\ \ 

The previous result provides an enlightening observation: If we seek to find,
within \textsf{ZFC}, an ideal that serves as a counterexample to the Category
Dichotomy, such an ideal cannot be non-meager and cannot be analytic. In this
sense, it cannot be too large (non-meager) nor too small (analytic).

\section{The Category Dichotomy for nowhere dense ideals
\label{seccion ideales nwd}}

Let $\left(  X,\tau\right)  $ be a topological space. By \textsf{nwd}$\left(
X,\tau\right)  $ (or simply \textsf{nwd}$\left(  X\right)  $ if the topology
is clear from context) we denote the ideal of nowhere dense subsets of
$\left(  X,\tau\right)  .$ As expected, our main interest is where the space
is countable. In \cite{NWDCommonDivisionTop}, \cite{NWDPositiveIntegers} and
\cite{RelationsNWDtop} the authors study the nowhere dense ideals for
different topologies with a number theoretic flavor. We will explore
properties of a topological space that will ensure its ideal of nowhere dense
sets fails Category Dichotomy. For the convenience of the reader, we now
review the main topological notions that will be needed in this section.

\begin{definition}
Let $\left(  X,\tau\right)  $ be a topological space, $b\in X$ and
$\mathcal{B}$ a family of non-empty open sets of $X.$

\begin{enumerate}
\item $\left(  X,\tau\right)  $ is \emph{zero} \emph{dimensional }if it has a
base of clopen sets.

\item \textsf{Clop}$\left(  X\right)  $ denotes the family of clopen subsets
of $X.$

\item Let $A\subseteq X$ be countable. $A$ \emph{converges to} $b$ (denoted by
$A\longrightarrow b$) if every open subset of $b$ almost contains $A.$

\item $\left(  X,\tau\right)  $ is \emph{Fr\'{e}chet} if for every $a\in X$
and $Y\subseteq X$ such that $a\in\overline{Y},$ there is $A\in\left[
Y\right]  ^{\omega}$ that converges to $a.$

\item $\mathcal{B}$ is a $\pi$\emph{-base }if every non-empty subset of $X$
contains an element of $\mathcal{B}.$

\item The $\pi$\emph{-weight of} $X$ is the smallest size of a $\pi$-base of
$X.$

\item $\mathcal{B}$ is a \emph{local} $\pi$\emph{-base at }$b$ if every
neighborhood of $b$ contains an element of $\mathcal{B}.$

\item The $\pi$\emph{-character of} $b$ is the smallest size of a \emph{local}
$\pi$\emph{-base at }$b.$

\item $X$ \emph{has uncountable} $\pi$\emph{-character everywhere }if every
point of $X$ has uncountable $\pi$-character.
\end{enumerate}
\end{definition}

\qquad\qquad\qquad\ \ 

Recall that the closure of a nowhere dense set is also nowhere dense. In this
way, \textsf{nwd}$\left(  X\right)  $ is an ideal generated by closed sets.

\begin{proposition}
Let $X$ be a countable, Fr\'{e}chet space with no isolated points. The ideal
\textsf{nwd}$\left(  X\right)  $ is meager.
\end{proposition}

\begin{proof}
We may assume that $X=\omega.$ Since our space is Fr\'{e}chet and has no
isolated points, for every $n\in\omega$, we can find $A_{n}\in\left[
\omega\right]  ^{\omega}$ that converges to $n.$ Now, find an interval
partition $P=\left\{  P_{n}\mid n\in\omega\right\}  $ such that $P_{n}\cap
A_{i}\neq\emptyset$ for every $n\in\omega$ and $i\leq n.$ It follows that
every $Y\subseteq\omega$ that contains infinitely many intervals of $P$, is
dense. In this way, $P$ is a Talagrand partition of \textsf{nwd}$\left(
X\right)  ,$ and by Theorem \ref{Teorema Talagrand}, we conclude that
\textsf{nwd}$\left(  X\right)  $ is meager.
\end{proof}

\qquad\qquad\qquad\qquad

By the last remark in the previous section, it is natural to seek a
counterexample to the Category Dichotomy within the class of nowhere dense
ideals for countable topologies.

\begin{proposition}
[Topological Disjoint Refinement Lemma]\label{top disjoint refinement}
\ \ \qquad\ 

Let $X$ be a zero dimensional space with uncountable $\pi$-character
everywhere. For every collection $\left\{  U_{n}\mid n\in\omega\right\}  $ of
non-empty open sets, there is a pairwise disjoint family of non-empty clopen
sets $\left\{  C_{n}\mid n\in\omega\right\}  $ such that $C_{n}\subseteq
U_{n}$ for every $n\in\omega.$
\end{proposition}

\begin{proof}
Since $X$ is zero dimensional, we may actually assume that each $U_{n}$ is
clopen. Pick any $a_{0}\in U_{0}.$ Since $\left\{  U_{n}\mid n\in
\omega\right\}  $ is not a local $\pi$-base at $a_{0},$ we can find $V_{0}\in$
\textsf{Clop}$\left(  X\right)  $ with $a_{0}\in V_{0}$ such that $V_{0}$ does
not contain any of the $U_{n}.$ Define $C_{0}=U_{0}\cap V_{0}.$ Denote
$U_{n}^{0}=U_{n}\setminus C_{0}.$ Pick any $a_{1}\in U_{1}^{0}.$ Since
$\left\{  U_{n}^{0}\mid n\in\omega\right\}  $ is not a local $\pi$-base at
$a_{1},$ we can find $V_{1}\in$ \textsf{Clop}$\left(  X\right)  $ with
$a_{1}\in V_{1}$ such that $V_{1}$ does not contain any of the $U_{n}^{0}.$
Define $C_{1}=U_{1}^{0}\cap V_{1}.$ We continue in this way and find each
$C_{n}.$
\end{proof}

\qquad\ \ \ \ \ \ \ \ \ 

With this result, we can prove:

\begin{proposition}
Let $X$ be a zero dimensional space that has uncountable $\pi$-character
everywhere. The ideal \textsf{nwd}$\left(  X\right)  $ is tight. In
particular, \textsf{nwd}$\left(  X\right)  $ $\nleqslant_{\text{\textsf{K}}}$
\textsf{nwd. }\label{nwdX no nwd}
\end{proposition}

\begin{proof}
Let $\left\{  Y_{n}\mid n\in\omega\right\}  \subseteq$ \textsf{nwd}$\left(
X\right)  ^{+},$ we need to find a nowhere dense set that intersects each
$Y_{n}.$ Since each $Y_{n}$ is not nowhere dense, we may find a non-empty open
set $U_{n}$ such that $Y_{n}$ is dense in $U_{n}.$ We now apply the
Topological Disjoint Refinement Lemma and find a pairwise disjoint family
$\left\{  C_{n}\mid n\in\omega\right\}  \subseteq$ \textsf{Clop}$\left(
X\right)  $ such that $\emptyset\neq C_{n}\subseteq U_{n}$ for every
$n\in\omega.$ Since $Y_{n}$ is dense in $U_{n},$ we may choose $y_{n}\in
Y_{n}\cap C_{n}.$ Let $D=\left\{  y_{n}\mid n\in\omega\right\}  $ and note
that $D\cap C_{n}=\left\{  y_{n}\right\}  $ for each $n\in\omega.$ In this
way, $D$ is a discrete set (hence nowhere dense) and it obviously has
non-empty intersection with each $Y_{n}.$
\end{proof}

\qquad\ \ \ 

We now want to find condition that guarantees that \textsf{nwd}$\left(
X\right)  $ is not Kat\v{e}tov above $\mathcal{ED}.$

\begin{lemma}
Let $X$ be a Fr\'{e}chet space with no isolated points, $\left\{  N_{n}\mid
n\in\omega\right\}  $ \ is a family of nowhere dense sets$,$ $a\in X$ and
$D\subseteq X$ a dense set. There is $S\in\left[  D\right]  ^{\omega}$ such
that: \label{suc evita nwd}

\begin{enumerate}
\item $S$ converges to $a.$

\item $\left\vert S\cap N_{n}\right\vert \leq1$ for every $n\in\omega.$
\end{enumerate}
\end{lemma}

\begin{proof}
Recall that the closure of a nowhere dense set is again nowhere dense, so we
may assume each $N_{n}$ is closed. Denote $M_{n}=%
{\textstyle\bigcup\limits_{i\leq n}}
N_{i}\in$ \textsf{nwd}$\left(  X\right)  $ and $M_{n}^{\text{\textsf{c}}%
}=X\setminus M_{n},$ which is an open dense set. $X$ is Fr\'{e}chet, $a$ is
not an isolated point and $D$ is dense, so we can find $B=\left\{  b_{n}\mid
n\in\omega\right\}  \in\left[  D\right]  ^{\omega}$ that converges to $a$. We
may assume $b_{n}\neq a$ for every $n\in\omega.$ Now, since $D\cap
M_{n}^{\text{\textsf{c}}}$ is dense (for each $n\in\omega$), we may find a
countable $C_{n}\subseteq$ $D\cap M_{n}^{\text{\textsf{c}}}$ converging to
$b_{n}.$ Define $C=%
{\textstyle\bigcup\limits_{n\in\omega}}
C_{n}.$ It is clear that $a\in\overline{C}.$ Once again since $X$ is
Fr\'{e}chet, we can find $E\in\left[  C\right]  ^{\omega}$ such that
$E\longrightarrow a.$

\qquad\qquad\ \ 

\begin{claim}
\ \ \ \ $E\cap N_{n}$ is finite for every $n\in\omega.$
\end{claim}

\qquad\ \ \ 

Since $C_{n}$ converges to $b_{n}$ and $b_{n}\neq a,$ it follows that $E\cap
C_{n}$ is finite. Since $E\cap N_{n}\subseteq%
{\textstyle\bigcup\limits_{i<n}}
E\cap C_{i},$ the claim follows.

\qquad\qquad\qquad\ \ 

We can now easily find $S\in\left[  E\right]  ^{\omega}$ that intersects each
$N_{n}$ in at most one point.
\end{proof}

\qquad\qquad\qquad\ \ \ 

We now prove the following:

\begin{proposition}
Let $X$ be a countable Fr\'{e}chet space with no isolated points. The ideal
\textsf{nwd}$\left(  X\right)  $ is weakly selective (it has no restrictions
Kat\v{e}tov above $\mathcal{ED}$).\label{nwdX no ED}
\end{proposition}

\begin{proof}
Let $Y\in$ \textsf{nwd}$\left(  X\right)  ^{+}$ and $\left\{  N_{n}\mid
n\in\omega\right\}  $ be a partition of $Y$ in nowhere dense sets. Since $Y$
is not nowhere dense, we can find a non-empty open set $U\subseteq\overline
{Y}.$ Note that $U$ is on its own a Fr\'{e}chet space with no isolated points
and each $N_{n}$ is nowhere dense in $U.$

\qquad\qquad\qquad\ \ \ \ 

Take an enumeration $U=\left\{  a_{n}\mid n\in\omega\right\}  .$ We now apply
Lemma \ref{suc evita nwd} infinitely many times and for each $n\in\omega,$ we
find $A_{n}$ $\subseteq Y$ with the following properties:

\begin{enumerate}
\item $A_{n}$ converges to $a_{n}.$

\item $\left\vert A_{n}\cap N_{m}\right\vert \leq1$ for every $m\in\omega.$
\end{enumerate}

\qquad\ \ \ \ \qquad\ \ 

We can now extract $B\in\left[  Y\right]  ^{\omega}$ such that for every
$n\in\omega,$ we have the following:

\begin{enumerate}
\item $B\cap A_{n}$ is infinite.

\item $B\cap N_{n}$ has at most one point.
\end{enumerate}

\qquad\qquad\ \ 

Finally, note that $U\subseteq\overline{B},$ so $B$ is not nowhere dense.
\end{proof}

\qquad\ \ \ \ 

By combining Proposition \ref{nwdX no nwd} and Proposition \ref{nwdX no ED},
we conclude the following:

\begin{theorem}
If $X$ is a topological space with the following properties:

\begin{enumerate}
\item $X$ is countable.

\item $X$ is zero dimensional.

\item $X$ is Fr\'{e}chet.

\item $X$ has uncountable $\pi$-character everywhere.
\end{enumerate}

Then, \textsf{nwd}$\left(  X\right)  $ $\nleqslant_{\text{\textsf{K}}}$
\textsf{nwd }and $\mathcal{ED}$ $\nleqslant_{\text{\textsf{K}}}$
\textsf{nwd}$\left(  X\right)  \upharpoonright Y$ for every $Y\in$
\textsf{nwd}$\left(  X\right)  ^{+}.$ In other words, \textsf{nwd}$\left(
X\right)  $ does not satisfy the Category Dichotomy.
\end{theorem}

\qquad\qquad\qquad\ \ 

Our quest for finding a counterexample for the Category Dichotomy will be
complete if we can find a space with the aforementioned properties.
Fortunately, a space with these qualities has already appeared in the
literature. In \cite{PiWeightandFrechetProperty}, the first author proved the
following theorem:

\begin{theorem}
[D.]There is a topological space $\mathbb{D}=\left(  \omega^{<\omega}%
,\tau_{\mathbb{D}}\right)  $ that is Fr\'{e}chet, zero dimensional and with
$\pi$-weight at least $\mathfrak{b}.$ \label{Dow space}
\end{theorem}

\qquad\qquad\qquad\ \ \ \ 

Although it is not explicitly stated in \cite{PiWeightandFrechetProperty} that
$\mathbb{D}$ has uncountable $\pi$-character everywhere, the proof that its
$\pi$-weight is at least $\mathfrak{b}$ applies at any point. We briefly
highlight the relevance and context of Theorem \ref{Dow space}. In
\cite{MalykhinProblem}, the fourth author and Ramos Garc\'{\i}a resolved an
old problem of Malykhin by constructing a model in which every countable
Fr\'{e}chet topological group is second countable. Juh\'{a}sz subsequently
questioned the role of algebra in this theorem. Since, in the context of
topological groups, second countability and having countable $\pi$-weight are
equivalent properties, Juh\'{a}sz raised the question of whether there exists
(in \textsf{ZFC}) a countable Fr\'{e}chet space with uncountable $\pi$-weight.
This problem remained unsolved until the first author proved Theorem
\ref{Dow space}.

\begin{corollary}
The Category Dichotomy fails for \textsf{nwd}$\left(  \mathbb{D}\right)  .$
\end{corollary}

\section{The Category Dichotomy for nowhere dense ideals induced by
independent families \label{seccion familias independientes}}

Although we have already seen an example of an ideal that does not satisfy the
Category Dichotomy, it is worthwhile to explore which classes of ideals do
satisfy it. Inspired by the results from the previous section, we aim to gain
a deeper understanding of when the ideal of nowhere dense sets of a countable
space satisfies the dichotomy. In this section, we will examine the ideals of
spaces generated by independent families. We begin by recalling the relevant concepts.

\qquad\qquad\qquad\qquad\ \ 

In this section, for a given $A\subseteq\omega,$ denote $A^{0}=A$ and
$A^{1}=\omega\setminus A.$ For a set $\mathcal{B},$ define $\mathbb{C}%
_{\mathcal{B}}$ as the set of all functions whose domain is a finite set of
$\mathcal{B}$ and its range is contained in $\left\{  0,1\right\}  $ (in other
words, $\mathbb{C}_{\mathcal{B}}$ is the standard forcing notion for adding a
Cohen subset of $\mathcal{B}$ with finite conditions).

\begin{definition}
Let $\mathcal{B}\subseteq\left[  \omega\right]  ^{\omega}$ infinite. We say
that $\mathcal{B}$ is an \emph{independent family }if for every $p\in
\mathbb{C}_{\mathcal{B}},$ we have that $%
{\textstyle\bigcap\limits_{A\in dom\left(  p\right)  }}
A^{p\left(  A\right)  }$ is infinite. The set $%
{\textstyle\bigcap\limits_{A\in dom\left(  p\right)  }}
A^{p\left(  A\right)  }$ will be denoted by $B^{p}.$
\end{definition}

\qquad\ \ \ \ \ \ \ \ 

A classical theorem of Fichtenholz and Kantorovich states that there exists a
perfect independent family (see \cite{Jech}). In particular, for any infinite
cardinality up to $\mathfrak{c}$, there exists an independent family of that
size. In recent years, there has been significant research on independent
families. Readers interested in learning more may refer to the papers
\cite{Con(u>i)}, \cite{DefinableMaximalIndependentFamilies},
\cite{FreeSequences}, \cite{OnPospisilIdeals}, \cite{PartitionForcing},
\cite{IdealsofIndependence}, \cite{CohenPreservationandIndependence} or
\cite{SpectrumofIndependence} for further details. The thesis \cite{Perron} is
a very good introduction to independent families and related topics.

\begin{definition}
Let $\mathcal{B}$ be an independent family.

\begin{enumerate}
\item We say that $\mathcal{B}$ \emph{separates points }if for every distinct
$m,n\in\omega,$ there is $B\in\mathcal{B}$ such that $m\in B$ and $n\notin B.$

\item The \emph{envelope of }$\mathcal{B}$ is defined as \textsf{Env}$\left(
\mathcal{B}\right)  =\{B^{p}\mid p\in\mathbb{C}_{\mathcal{B}}\}.$\qquad\ \ \ \ \ \ \ \ \ 
\end{enumerate}
\end{definition}

Letting $\mathcal{B}$ be an independent family that separates points, we can
use \textsf{Env}$\left(  \mathcal{B}\right)  $ to generated a topology. The
space $X\left(  \mathcal{B}\right)  =\left(  \omega,\tau_{\mathcal{B}}\right)
$ is the topological space that has \textsf{Env}$\left(  \mathcal{B}\right)  $
as a base. The following lemma is well-known, the reader may consult
\cite{Perron} for a proof.

\begin{lemma}
Let $\mathcal{B}$ be an independent family that separates points.

\begin{enumerate}
\item $X\left(  \mathcal{B}\right)  $ is a Hausdorff, countable, zero
dimensional space with no isolated points.

\item The following are equivalent:

\begin{enumerate}
\item $\mathcal{B}$ is countable.

\item $X\left(  \mathcal{B}\right)  $ is second countable.

\item $X\left(  \mathcal{B}\right)  $ is homeomorphic to the rational numbers.
\end{enumerate}

\item Let $A\subseteq\omega$ such that $A\notin\mathcal{B}.$ The following are equivalent:

\begin{enumerate}
\item $\mathcal{B}\cup\left\{  A\right\}  $ is independent.

\item Both $A$ and $\omega\setminus A$ are dense in $X\left(  \mathcal{B}%
\right)  .$
\end{enumerate}
\end{enumerate}
\end{lemma}

\qquad\qquad\qquad\ \ \ \ 

In case $\mathcal{B}$ is a maximal independent family, the space $X\left(
\mathcal{B}\right)  $ does not have disjoint dense sets. Such spaces are
called \emph{irresolvable. }The reader may consult \cite{irr} to learn more
about irresolvable spaces and their cardinal invariants. To avoid constant
repetition, \textbf{from now on we assume all our independent families
separates points.}

\begin{lemma}
Let $\mathcal{B}$ be an uncountable independent family. $X\left(
\mathcal{B}\right)  $ has uncountable $\pi$-character everywhere. In fact, the
$\pi$-character of every point is $\left\vert \mathcal{B}\right\vert .$
\end{lemma}

\begin{proof}
Let $\kappa=\left\vert \mathcal{B}\right\vert .$ Since \textsf{Env}$\left(
\mathcal{B}\right)  $ has the same size as $\mathcal{B},$ it follows that the
$\pi$-character of any point is at least $\kappa.$ We now proceed by
contradiction, assume there is $a\in\omega$ and $\mathcal{U}$ a local $\pi
$-base of $a$ with $\mu=\left\vert \mathcal{U}\right\vert <$ $\kappa.$ Note
that we may assume that $\mathcal{U\subseteq}$ \textsf{Env}$\left(
\mathcal{B}\right)  ,$ so there is $P\subseteq\mathbb{C}_{\mathcal{B}}$ of
size $\mu$ such that $\mathcal{U}=\left\{  B^{p}\mid p\in P\right\}  .$ Since
$\mu<\kappa,$ there is $A\in\mathcal{B}$ such that $A\notin$ \textsf{dom}%
$\left(  p\right)  $ for every $p\in P.$ Find $i<2$ such that $a\in A^{i}.$ It
follows that $A^{i}$ is an open neighborhood of $a,$ so there should be $p\in
P$ such that $B^{p}\subseteq A^{i}.$ However, this contradicts that
$\mathcal{B}$ is an independent family.
\end{proof}

\qquad\qquad\qquad\qquad\ \ 

With Proposition \ref{nwdX no nwd} we conclude the following:

\begin{corollary}
If $\mathcal{B}$ is an uncountable independent family, then \textsf{nwd}%
$\left(  X\left(  \mathcal{B}\right)  \right)  $ is tight. In particular, it
is not Kat\v{e}tov below \textsf{nwd. }\label{independientes no nwd}
\end{corollary}

We introduce the following notation:

\begin{definition}
Let $\mathcal{P}\subseteq\left[  \omega\right]  ^{\omega}.$ Define
\textsf{Hit}$\left(  \mathcal{P}\right)  $ as the set of all $X\subseteq
\omega$ that have infinite intersection with every element of $\mathcal{P}.$
\end{definition}

Note that if $\mathcal{F}$ is a filter, then \textsf{Hit}$\left(
\mathcal{F}\right)  $ is just $\mathcal{F}^{+}.$ We will recall the following
result, which is Proposition 6.24 in \cite{HandbookBlass}.

\begin{proposition}
Let $\mathcal{P\subseteq}$ $\left[  \omega\right]  ^{\omega}$ of size less
than $\mathfrak{d}$ and\ $\mathcal{D}\subseteq$ \textsf{Hit}$\left(
\mathcal{P}\right)  $ a countable and decreasing family. There is $A\in$
\textsf{Hit}$\left(  \mathcal{P}\right)  $ that is a pseudointersection of
$\mathcal{D}.$ \label{Prop d Blass}
\end{proposition}

We have the following characterization of the dominating number:

\begin{theorem}
The cardinal $\mathfrak{d}$ is the least size of an independent family
$\mathcal{B}$ such that \textsf{nwd}$\left(  X\left(  \mathcal{B}\right)
\right)  $ is not \textsf{P}$^{-}.$ \label{independiente p-}
\end{theorem}

\begin{proof}
Let $\mathcal{B}\subseteq\left[  \omega\right]  ^{\omega}$ be an independent
family of size less than $\mathfrak{d},$ we will prove that \textsf{nwd}%
$\left(  X\left(  \mathcal{B}\right)  \right)  $ is \textsf{P}$^{-}.$ Let
$Y\in$ \textsf{nwd}$\left(  X\left(  \mathcal{B}\right)  \right)  ^{+}$ and
$\mathcal{D}\subseteq$ $($\textsf{nwd}$\left(  X\left(  \mathcal{B}\right)
\right)  \upharpoonright Y)^{\ast}$ countable and decreasing. We need to find
a pseudointersection in \textsf{nwd}$\left(  X\left(  \mathcal{B}\right)
\right)  ^{+}.$ Since $Y$ is somewhere dense, we may find $p\in\mathbb{C}%
_{\mathcal{B}}$ such that $Y$ is dense in $B^{p}.$ Let $\mathcal{P=}%
\{B^{q}\cap Y\mid q\in\mathbb{C}_{\mathcal{B}}\wedge p\subseteq q\}.$ It
follows that $\mathcal{D\in}$\textsf{Hit}$\left(  \mathcal{P}\right)  $ and
$\mathcal{P}$ has size less than $\mathfrak{d}$. We can then invoke
Proposition \ref{Prop d Blass} and obtain $A\in$ $\mathcal{P}^{+}$ a
pseudointersection of $\mathcal{D}.$ It follows that $A$ is dense in $B^{p}$,
so $A\in$ \textsf{nwd}$\left(  X\left(  \mathcal{B}\right)  \right)  ^{+}.$

\qquad\qquad\qquad\qquad\qquad

We now need to find an independent family of size $\mathfrak{d}$ whose nowhere
dense ideal is Kat\v{e}tov above \textsf{fin}$\times$\textsf{fin. }For
every\textsf{ }$n\in\omega,$ we denote the column $C_{n}=\left\{  n\right\}
\times\omega.$ Fix $\mathcal{A}=\left\{  A_{\alpha}\mid\alpha<\mathfrak{d}%
\right\}  $ an independent family such that for every $n\in\omega,$ there are
infinitely many elements of $\mathcal{A}$ for which $n$ belongs and
$\mathcal{D}=\left\{  f_{\alpha}\mid\alpha<\mathfrak{d}\right\}
\subseteq\omega^{\omega}$ a $\leq$-dominating family (not only with
$\leq^{\ast}$). We fuse $\mathcal{A}$ and $\mathcal{D}$ as follows: for any
$\alpha<\mathfrak{d}$, define $B_{\alpha}=\left\{  \left(  n,m\right)  \mid
n\in A_{\alpha}\wedge f_{\alpha}\left(  n\right)  <m\right\}  .$ Let
$\mathcal{B}=\left\{  B_{\alpha}\mid\alpha<\mathfrak{d}\right\}  .$

\begin{claim}
$\mathcal{B}$ is an independent family.
\end{claim}

\qquad\qquad\qquad\ \ \ 

Let $F,G\subseteq\mathfrak{d}$ finite and disjoint. We need to prove that $Y=$
$%
{\textstyle\bigcap\limits_{\alpha\in F}}
B_{\alpha}\cap%
{\textstyle\bigcap\limits_{\beta\in G}}
\left(  \omega\times\omega\setminus B_{\beta}\right)  $ is infinite. Since
$\mathcal{A}$ is independent, we can find $k\in$ $%
{\textstyle\bigcap\limits_{\alpha\in F}}
A_{\alpha}\cap%
{\textstyle\bigcap\limits_{\beta\in G}}
\left(  \omega\setminus A_{\beta}\right)  .$ It follows that the column
$C_{k}$ is almost contained in $Y.$ This finishes the claim.

\qquad\ \ 

We will prove that each column $C_{n}$ is nowhere dense. Let $p\in
\mathbb{C}_{\mathcal{B}}$ and we look at the open set $B^{p}.$ Pick any
$\alpha<\mathfrak{d}$ such that $B_{\alpha}\notin$ \textsf{dom}$\left(
p\right)  $ and $n\in B_{\alpha}.$ Define $q=p\cup\left\{  \left(
\alpha,0\right)  \right\}  .$ It follows that $B^{q}\subseteq B^{p}$ and
$B^{q}\cap C_{n}$ is finite, so we are done.

\qquad\qquad\qquad

Let $\alpha<\mathfrak{d,}$ we will now prove that $G=\left\{  \left(
n,m\right)  \mid m\leq f_{\alpha}\left(  n\right)  \right\}  $ is nowhere
dense. Let $p\in\mathbb{C}_{\mathcal{B}}$ and we look at the open set $B^{p}.$
Pick any $\beta<\mathfrak{d}$ such that $B_{\beta}\notin$ \textsf{dom}$\left(
p\right)  $ and $f_{\alpha}<f_{\beta}$ for each $\alpha\in$ \textsf{dom}%
$\left(  p\right)  $ (this is possible since $\mathcal{D}$ is a dominating
family). Define $q=p\cup\left\{  \left(  \beta,0\right)  \right\}  .$ It
follows that $B^{q}\subseteq B^{p}$ and $B^{q}\cap G=\emptyset$, so we are done.

\qquad\qquad\qquad

Finally, by the previous two remarks and the fact that $\mathcal{D}$ is a
dominating family, we conclude that \textsf{fin}$\times$\textsf{fin
}$\subseteq$ \textsf{nwd}$\left(  X\left(  B\right)  \right)  .$ We rejected
to verify that $\mathcal{B}$ separates points. However, if this was not the
case, we can perform finite changes to countably many elements of
$\mathcal{B}$ and make it separate points. The arguments described above still
work after this finite modifications.
\end{proof}

\qquad\ \ \ \ 

On the other hand, we have the following:

\begin{proposition}
Let $\mathcal{B}$ be an independent family. If $\left\vert \mathcal{B}%
\right\vert <$ \textsf{non*}$\mathsf{(}\mathcal{ED}_{\text{\textsf{fin}}}),$
then \textsf{nwd}$\left(  \mathcal{B}\right)  $ is \textsf{Q}$^{+}.$
\label{independiente Q+}
\end{proposition}

\begin{proof}
Let $\mathcal{B}\subseteq\left[  \omega\right]  ^{\omega}$ be an independent
family of size $\kappa<$ \textsf{non*}$\mathsf{(}\mathcal{ED}%
_{\text{\textsf{fin}}})$, $Y\in$ \textsf{nwd}$\left(  X\left(  \mathcal{B}%
\right)  \right)  ^{+}$ and $\mathcal{P=}\left\{  P_{n}\mid n\in
\omega\right\}  $ an interval partition of $Y.$ We need to find a positive
partial selector. We may assume the intervals of $\mathcal{P}$ are increasing
in size. Since $Y$ is not nowhere dense, we may find $p\in\mathbb{C}%
_{\mathcal{B}}$ such that $Y$ is dense in $B^{p}.$ Let $\mathcal{U=}%
\{B^{q}\cap Y\mid q\in\mathbb{C}_{\mathcal{B}}\wedge p\subseteq q\}.$ For
every $U\in\mathcal{U},$ we find $f_{U}$ a partial infinite function such that
for every $n\in$ \textsf{dom}$\left(  f_{U}\right)  ,$ we have that $f_{U}\in
P_{n}\cap U.$ Since $\left\vert \mathcal{U}\right\vert <$ \textsf{non*}%
$\mathsf{(}\mathcal{ED}_{\text{\textsf{fin}}}),$ we can find a function
$h\in\omega^{\omega}$ such that $h\left(  n\right)  \in P_{n}$ for every
$n\in\omega$ and has infinite intersection with each $f_{U}$ for
$U\in\mathcal{U}.$ It follows that the image of $h$ is dense in $B^{p}$ and it
is a selector, so we are done.
\end{proof}

\qquad

We do not know the following:

\begin{problem}
Is there an independent family $\mathcal{B}$ of size \textsf{non*}%
$\mathsf{(}\mathcal{ED}_{\text{\textsf{fin}}})$ such that \textsf{nwd}$\left(
\mathcal{B}\right)  $ is not \textsf{Q}$^{+}$?
\end{problem}

\qquad\qquad\ \ \ 

We can now prove:

\begin{theorem}
If $\omega_{1}<$ \textsf{cov}$\left(  \mathcal{M}\right)  ,$ then there is
$\mathcal{B}$ an independent family such that \textsf{nwd}$\left(  X\left(
\mathcal{B}\right)  \right)  $ does not satisfy the Category Dichotomy.
\end{theorem}

\begin{proof}
Let $\mathcal{B}$ be an independent family of size $\omega_{1}.$ Since it is
uncountable, by Corollary \ref{independientes no nwd}, we know that
\textsf{nwd}$\left(  X\left(  \mathcal{B}\right)  \right)  \nleq
_{\text{\textsf{K}}}$ \textsf{nwd. }Finally, the other possibility of the
dichotomy is impossible since \textsf{cov}$\left(  \mathcal{M}\right)  =$
\textsf{min}$\mathsf{(}\mathfrak{d,}$\textsf{non*}$\mathsf{(}\mathcal{ED}%
_{\text{\textsf{fin}}}))$ and by Proposition \ref{independiente Q+} and
Proposition \ref{independiente p-}.
\end{proof}

\qquad\ \ \ 

The following problem remains open:

\begin{problem}
Is it consistent that the Category Dichotomy holds for the nowhere dense
ideals of topologies generated by independent families?
\end{problem}

\qquad\qquad\ \ \ \qquad\ \ 

Building on the results of this section and the previous one, it is natural to
ask the following questions:

\begin{problem}
Let $\mathcal{B}$ be an uncountable independent family. When is $X\left(
\mathcal{B}\right)  $ Fr\'{e}chet? Is there such a family in \textsf{ZFC}?
\end{problem}

\qquad\ \ \ \ 

\begin{acknowledgement}
The second author wishes to thank the Set Theory Seminar at Universidad de los
Andes (Bogot\'{a}) for their generous feedback. We would like to thank David
Valderrama, Ramiro de la Vega, and Sebastian Rodriguez for their valuable
comments and contributions.
\end{acknowledgement}

{\normalsize
\bibliographystyle{plain}
\bibliography{Bibliografia}

\def\cprime{$'$}
\begin{thebibliography}{10}

\bibitem{InvariancePropertiesofAlmostDisjointFamilies}
M.~Arciga-Alejandre, M.~Hru{\v{s}}{\'a}k, and C.~Martinez-Ranero.
\newblock Invariance properties of almost disjoint families.
\newblock {\em J. Symbolic Logic}, 78(3):989--999, 2013.

\bibitem{KatetovOrderImply}
Pawe\l Barbarski, Rafa\l Filip\'{o}w, Nikodem Mro\.{z}ek, and Piotr Szuca.
\newblock When does the {K}at\v{e}tov order imply that one ideal extends the other?
\newblock {\em Colloq. Math.}, 130(1):91--102, 2013.

\bibitem{Barty}
Tomek {Bartoszy\'nski} and Haim {Judah}.
\newblock {\em {Set theory: on the structure of the real line.}}
\newblock Wellesley, MA: A. K. Peters Ltd., 1995.

\bibitem{HandbookBlass}
Andreas Blass.
\newblock Combinatorial cardinal characteristics of the continuum.
\newblock In Matthew Foreman and Akihiro Kanamori, editors, {\em Handbook of set theory. {V}ols. 1, 2, 3}, pages 395--489. Springer, Dordrecht, 2010.

\bibitem{InequivalentGrowthTypes}
Andreas Blass and Claude Laflamme.
\newblock Consistency results about filters and the number of inequivalent growth types.
\newblock {\em J. Symbolic Logic}, 54(1):50--56, 1989.

\bibitem{Brendlenowheredense}
J\"org Brendle.
\newblock Between {$P$}-points and nowhere dense ultrafilters.
\newblock {\em Israel J. Math.}, 113:205--230, 1999.

\bibitem{DefinableMaximalIndependentFamilies}
J\"{o}rg Brendle, Vera Fischer, and Yurii Khomskii.
\newblock Definable maximal independent families.
\newblock {\em Proc. Amer. Math. Soc.}, 147(8):3547--3557, 2019.

\bibitem{BrendleFlaskova}
J\"org Brendle and Jana Fla{\v{s}}kov\'a.
\newblock Generic existence of ultrafilters on the natural numbers.
\newblock {\em Fund. Math.}, 236(3):201--245, 2017.

\bibitem{Ultrafiltersonomega}
J\"{o}rg Brendle and Saharon Shelah.
\newblock Ultrafilters on {$\omega$}---their ideals and their cardinal characteristics.
\newblock {\em Trans. Amer. Math. Soc.}, 351(7):2643--2674, 1999.

\bibitem{ForcingIndestructibilityofMADFamilies}
J\"org Brendle and Shunsuke Yatabe.
\newblock Forcing indestructibility of {MAD} families.
\newblock {\em Ann. Pure Appl. Logic}, 132(2-3):271--312, 2005.

\bibitem{AboveFsigma}
J.~Cancino-Manr\'{\i}quez.
\newblock Every maximal ideal may be {K}at\v{e}tov above of all {$F_\sigma$} ideals.
\newblock {\em Trans. Amer. Math. Soc.}, 375(3):1861--1881, 2022.

\bibitem{irr}
Jonathan {Cancino-Manr{\'\i}quez}, Michael {Hru\v{s}\'ak}, and David {Meza-Alc\'antara}.
\newblock {Countable irresolvable spaces and cardinal invariants.}
\newblock {\em {Topol. Proc.}}, 44:189--196, 2014.

\bibitem{FreeSequences}
David Chodounsk\'{y}, Vera Fischer, and Jan Greb\'{i}k.
\newblock Free sequences in {${P}(\omega)/{fin}$}.
\newblock {\em Arch. Math. Logic}, 58(7-8):1035--1051, 2019.

\bibitem{TherearenoPpointsinSilverExtensions}
David Chodounsk\'{y} and Osvaldo Guzm\'{a}n.
\newblock There are no {P}-points in {S}ilver extensions.
\newblock {\em Israel J. Math.}, 232(2):759--773, 2019.

\bibitem{SurveyDestructibility}
David Chodounsk\'{y} and Osvaldo Guzm\'{a}n.
\newblock Indestructibility of ideals and {MAD} families.
\newblock {\em Ann. Pure Appl. Logic}, 172(5):Paper No. 102905, 20, 2021.

\bibitem{PartitionForcing}
Jorge~A. Cruz-Chapital, Vera Fischer, Osvaldo Guzm\'{a}n, and Jaroslav \v{S}upina.
\newblock Partition forcing and independent families.
\newblock {\em J. Symb. Log.}, 88(4):1590--1612, 2023.

\bibitem{PiWeightandFrechetProperty}
Alan Dow.
\newblock {$\pi$}-weight and the {F}r\'{e}chet-{U}rysohn property.
\newblock {\em Topology Appl.}, 174:56--61, 2014.

\bibitem{IliasBook}
Ilijas Farah.
\newblock Analytic quotients: theory of liftings for quotients over analytic ideals on the integers.
\newblock {\em Mem. Amer. Math. Soc.}, 148(702):xvi+177, 2000.

\bibitem{KatetovHindman}
Rafa\l Filip\'{o}w, Krzysztof Kowitz, and Adam Kwela.
\newblock Kat\v{e}tov order between {H}indman, {R}amsey and summable ideals.
\newblock {\em Arch. Math. Logic}, 63(7-8):859--876, 2024.

\bibitem{IdealsofIndependence}
Vera Fischer and Diana~Carolina Montoya.
\newblock Ideals of independence.
\newblock {\em Arch. Math. Logic}, 58(5-6):767--785, 2019.

\bibitem{SpectrumofIndependence}
Vera Fischer and Saharon Shelah.
\newblock The spectrum of independence.
\newblock {\em Arch. Math. Logic}, 58(7-8):877--884, 2019.

\bibitem{CohenPreservationandIndependence}
Vera Fischer and Corey~Bacal Switzer.
\newblock Cohen preservation and independence.
\newblock {\em Ann. Pure Appl. Logic}, 174(8):Paper No. 103291, 11, 2023.

\bibitem{UltrafiltersMathematics}
Isaac Goldbring.
\newblock {\em Ultrafilters throughout mathematics}, volume 220 of {\em Graduate Studies in Mathematics}.
\newblock American Mathematical Society, Providence, RI, [2022] \copyright 2022.

\bibitem{KatetovMAD}
Osvaldo Guzm\'{a}n.
\newblock Kat\v{e}tov order on {MAD} families.
\newblock {\em J. Symb. Log.}, 89(2):794--828, 2024.

\bibitem{OnPospisilIdeals}
Osvaldo Guzm\'{a}n and Michael Hru\v{s}\'{a}k.
\newblock On {P}osp\'{\i}\v{s}il ideals.
\newblock {\em Topology Appl.}, 259:242--250, 2019.

\bibitem{StructuralKatetov}
Osvaldo Guzm\'{a}n-Gonz\'{a}lez and David Meza-Alc\'{a}ntara.
\newblock Some structural aspects of the {K}at\v{e}tov order on {B}orel ideals.
\newblock {\em Order}, 33(2):189--194, 2016.

\bibitem{CardinalInvariantsofAnalyticPIdeals}
Fernando Hern\'andez-Hern\'andez and Michael Hru{\v{s}}{\'a}k.
\newblock Cardinal invariants of analytic {$P$}-ideals.
\newblock {\em Canad. J. Math.}, 59(3):575--595, 2007.

\bibitem{HongZhang}
Jianyong Hong and Shuguo Zhang.
\newblock Cardinal invariants related to the i-ultrafilters.
\newblock {\em Science China Mathematics}, 43(1):1--6, 2013.

\bibitem{SelectivityofAlmostDisjointFamilies}
Michael Hru{\v{s}}{\'a}k.
\newblock Selectivity of almost disjoint families.
\newblock {\em Acta Univ. Carolin. Math. Phys.}, 41(2):13--21, 2000.

\bibitem{CombinatoricsofFiltersandIdeals}
Michael Hru{\v{s}}{\'a}k.
\newblock Combinatorics of filters and ideals.
\newblock In L.~Babinkostova, A.~E. Caicedo, S.~Geschke, and M.~Scheepers, editors, {\em Set theory and its applications}, volume 533 of {\em Contemp. Math.}, pages 29--69. Amer. Math. Soc., Providence, RI, 2011.

\bibitem{AlmostDisjointFamiliesandTopology}
Michael Hru{\v{s}}{\'a}k.
\newblock Almost disjoint families and topology.
\newblock In K.~P. Hart, J.~van Mill, and P.~Simon, editors, {\em Recent progress in general topology. {III}}, pages 601--638. Atlantis Press, Paris, 2014.

\bibitem{KatetovOrderonBorelIdeals}
Michael Hru{\v{s}}{\'a}k.
\newblock Kat{\v{e}}tov order on {B}orel ideals.
\newblock {\em Archive for Mathematical Logic}, 56(7):831--847, Nov 2017.

\bibitem{OrderingMADFamiliesalaKatetov}
Michael Hru{\v{s}}{\'a}k and Salvador Garc{\'i}a~Ferreira.
\newblock Ordering {MAD} families a la {K}at{\v{e}}tov.
\newblock {\em J. Symbolic Logic}, 68(4):1337--1353, 2003.

\bibitem{PairSplitting}
Michael Hru{\v{s}}{\'a}k, David Meza-Alc{\'a}ntara, and Hiroaki Minami.
\newblock Pair-splitting, pair-reaping and cardinal invariants of f σ -ideals.
\newblock {\em J. Symbolic Logic}, 75(2):661--677, 06 2010.

\bibitem{ForcingwithQuotients}
Michael Hru{\v{s}}{\'a}k and Jind{\v{r}}ich Zapletal.
\newblock Forcing with quotients.
\newblock {\em Arch. Math. Logic}, 47(7-8):719--739, 2008.

\bibitem{MalykhinProblem}
M.~Hru\v{s}\'{a}k and U.~A. Ramos-Garc\'{i}a.
\newblock Malykhin's problem.
\newblock {\em Adv. Math.}, 262:193--212, 2014.

\bibitem{Jech}
Thomas Jech.
\newblock {\em Set theory}.
\newblock Springer Monographs in Mathematics. Springer-Verlag, Berlin, 2003.
\newblock The third millennium edition, revised and expanded.

\bibitem{KatetovOriginal}
Miroslav Kat\v{e}tov.
\newblock Products of filters.
\newblock {\em Comment. Math. Univ. Carolinae}, 9:173--189, 1968.

\bibitem{Kechris}
Alexander~S. Kechris.
\newblock {\em Classical descriptive set theory}, volume 156 of {\em Graduate Texts in Mathematics}.
\newblock Springer-Verlag, New York, 1995.

\bibitem{KwelaUltrafiltersKatetov}
Krzysztof Kowitz and Adam Kwela.
\newblock Ultrafilters and the {K}at\v{e}tov order.
\newblock {\em Topology Appl.}, 361:Paper No. 109191, 2025.

\bibitem{KunenSomePoints}
Kenneth Kunen.
\newblock Some points in {$\beta N$}.
\newblock {\em Math. Proc. Cambridge Philos. Soc.}, 80(3):385--398, 1976.

\bibitem{Kunen}
Kenneth Kunen.
\newblock {\em Set theory}, volume~34 of {\em Studies in Logic (London)}.
\newblock College Publications, London, 2011.

\bibitem{KurilicMAD}
Milo\v s~S. Kurili\'c.
\newblock Cohen-stable families of subsets of integers.
\newblock {\em J. Symbolic Logic}, 66(1):257--270, 2001.

\bibitem{KwelaExtendability}
Adam Kwela.
\newblock On extendability to {$F_\sigma$} ideals.
\newblock {\em Arch. Math. Logic}, 61(7-8):881--890, 2022.

\bibitem{NWDCommonDivisionTop}
M.~Kwela.
\newblock Some properties of the ideal of nowhere dense sets in the common division topology.
\newblock {\em Acta Math. Hungar.}, 174(2):299--311, 2024.

\bibitem{NWDPositiveIntegers}
Marta Kwela and Andrzej Nowik.
\newblock Ideals of nowhere dense sets in some topologies on positive integers.
\newblock {\em Topology Appl.}, 248:149--163, 2018.

\bibitem{RudinBlassOrderingofUltrafilters}
Claude Laflamme and Jian-Ping Zhu.
\newblock The {R}udin-{B}lass ordering of ultrafilters.
\newblock {\em J. Symbolic Logic}, 63(2):584--592, 1998.

\bibitem{Happyfamilies}
A.~R.~D. Mathias.
\newblock Happy families.
\newblock {\em Ann. Math. Logic}, 12(1):59--111, 1977.

\bibitem{RamseyTypePropertiesofIdeals}
Hru{\v{s}}{\'a}k Michael, David Meza-Alc\'antara, Egbert Th\"ummel, and Carlos. Uzc\'ategui.
\newblock Ramsey type properties of ideals.
\newblock {\em Ann. Pure Appl. Logic}, 168(11):2022--2049, 2017.

\bibitem{NoQPOintsLaverModel}
Arnold~W. Miller.
\newblock There are no {$Q$}-points in {L}aver's model for the {B}orel conjecture.
\newblock {\em Proc. Amer. Math. Soc.}, 78(1):103--106, 1980.

\bibitem{Rat}
Arnold~W. Miller.
\newblock Rational perfect set forcing.
\newblock In {\em Axiomatic set theory ({B}oulder, {C}olo., 1983)}, volume~31 of {\em Contemp. Math.}, pages 143--159. Amer. Math. Soc., Providence, RI, 1984.

\bibitem{KatetovandKatetovBlassOrdersFsigmaIdeals}
Hiroaki Minami and Hiroshi Sakai.
\newblock Kat\v etov and {K}at\v etov-{B}lass orders on {$F_\sigma$}-ideals.
\newblock {\em Arch. Math. Logic}, 55(7-8):883--898, 2016.

\bibitem{RelationsNWDtop}
Andrzej Nowik and Paulina Szyszkowska.
\newblock On some relations between ideals of nowhere dense sets in topologies on positive integers.
\newblock {\em Period. Math. Hungar.}, 85(1):164--170, 2022.

\bibitem{Perron}
Michael~J. Perron.
\newblock {\em On the {S}tructure of {I}ndependent {F}amilies}.
\newblock ProQuest LLC, Ann Arbor, MI, 2017.
\newblock Thesis (Ph.D.)--Ohio University.

\bibitem{SakaiKatetov}
Hiroshi Sakai.
\newblock On {K}at{\v{e}}tov and {K}at{\v{e}}tov--{B}lass orders on analytic {P}-ideals and {B}orel ideals.
\newblock {\em Archive for Mathematical Logic}, 57(3):317--327, May 2018.

\bibitem{Con(u>i)}
Saharon {Shelah}.
\newblock {CON($\mathfrak u>\mathfrak i$).}
\newblock {\em {Arch. Math. Logic}}, 31(6):433--443, 1992.

\bibitem{ProperandImproper}
Saharon Shelah.
\newblock {\em Proper and {I}mproper {F}orcing}.
\newblock Perspectives in Mathematical Logic. Springer-Verlag, Berlin, second edition, 1998.

\bibitem{ShelahNowheredense}
Saharon Shelah.
\newblock There may be no nowhere dense ultrafilter.
\newblock In {\em Logic {C}olloquium '95 ({H}aifa)}, volume~11 of {\em Lecture Notes Logic}, pages 305--324. Springer, Berlin, 1998.

\bibitem{RamseySpaces}
Stevo Todorcevic.
\newblock {\em Introduction to {R}amsey spaces}, volume 174 of {\em Annals of Mathematics Studies}.
\newblock Princeton University Press, Princeton, NJ, 2010.

\end{thebibliography}
}

\qquad\qquad\qquad\ \ \ \ \ \ \ \ \ \ \ \ \ \ \ \ \ \ \ \qquad\qquad\qquad\ \ \ 

Alan Dow

Department of Mathematics and Statistics, UNC Charlotte

adow@charlotte.edu

\qquad\qquad\qquad\ \ \ \ \ \ \ \ \ \ \ \ \qquad\ \ \ \ \ \qquad\qquad\ \ \ 

Ra\'{u}l Figueroa-Sierra

Departamento de Matem\'{a}ticas, Universidad de Los Andes (Bogotá)

r.figueroa@uniandes.edu.co

\qquad\qquad\qquad\ \ \ \ \ \ \ \ \ \ \ \ \qquad\ \ \ \ \ \qquad\qquad\ \ \ 

Osvaldo Guzm\'{a}n

Centro de Ciencias Matem\'{a}ticas, UNAM.

oguzman@matmor.unam.mx

\qquad\qquad\qquad\ \ \ \ \ \ \ \qquad\qquad\qquad\ \ \ \ \ \ 

Michael Hru\v{s}\'{a}k

Centro de Ciencias Matem\'{a}ticas, UNAM.

michael@matmor.unam.mx

\end{document}